\documentclass[sn-mathphys]{sn-jnl}

\usepackage{graphicx}%
\usepackage{multirow}%
\usepackage{amsmath,amssymb,amsfonts}%
\usepackage{amsthm}%
\usepackage{mathrsfs}%
\usepackage[title]{appendix}%
\usepackage{xcolor}%
\usepackage{textcomp}%
\usepackage{manyfoot}%
\usepackage{booktabs}%
\usepackage{algorithm}%
\usepackage{algorithmicx}%
\usepackage{algpseudocode}%
\usepackage{listings}%

\usepackage[framemethod=tikz]{mdframed} 
\usepackage{circuitikz}
\usetikzlibrary{decorations.pathmorphing, patterns,shapes}
\usetikzlibrary{arrows}
\usetikzlibrary{patterns}
\usepackage{xcolor}
\usepackage{caption}
\usepackage{subcaption}
\usepackage{orcidlink}

\theoremstyle{thmstyleone}%
\newtheorem{theorem}{Theorem}
\newtheorem{proposition}[theorem]{Proposition}%

\theoremstyle{thmstyletwo}%
\newtheorem{example}{Example}%
\newtheorem{remark}{Remark}%

\theoremstyle{thmstylethree}%
\newtheorem{definition}{Definition}%
\newtheorem{con}{Conjecture}%
\newtheorem{corollary}{Corollary}%
\begin{document}

\title[Limiting Resistances]{Limiting Behavior of Resistances in Triangular Graphs}


\author*{\fnm{Russell} \sur{Hendel}}
\email{RHendel@Towson.Edu}



\affil*{\orgdiv{Mathematics}, \orgname{Towson University}, \orgaddress{
\city{Towson}, \postcode{21252}, \state{MD.}, \country{U.S.A.}}}




\abstract{\cite{Barrett9} studied  resistance labels of electrical circuits whose underlying graphs when embedded in the Cartesian plane has the form of an $n$-grid, $n$ rows of upright triangles. Proofs in \cite{Barrett9} introduced a row-reduction algorithm  which uses series, $\Delta$--Y, and Y--$\Delta$ electric transformations to transform an $n$-grid into an $n-1$ grid with equivalent resistances between specified nodes. This paper explores this row-reduction algorithm computationally.  The introductory part of the paper presents several conjectures  supported by  numerical evidence, showing that repeated application of the row-reduction algorithm to an initial $n$-grid uniformly labeled 1 asymptotically produces  triangular grids whose sides are labeled with rational multiples of $\frac{1}{e};$ moreover, the ratio of specified consecutive edges in the row-reduced grids  are asymptotically described  by four rational functions. The  main part of this paper studies a family of graphs whose edge labels are determined  using these limiting edge-ratios functions arising in the conjectures.     The main result proven is that these  $n$-grids and their repeated reductions under the row-reduction algorithm   possess vertical and rotational symmetries and satisfy the relationships captured by the four edge-ratio functions. Thus, the limiting edge-ratio relationships  are local algebraic relationships mirroring the global vertical and rotational symmetries possessed by the underlying graph. Additionally, because row-reduction is local (in contrast to the combinatoric Laplacian which is global) the paper is able to  introduce a mechanical verification method of proof for assertions about effective resistance identities. }




\vspace{2mm}

\maketitle


\section{Background and Goals}\label{sec:intro}

Several papers study resistance properties
of graphical configurations where edges are labeled with one-ohm resistances 
\cite{Barrett9,Barrett0,Barrett0b}.
These configurations are of interest in 
circuit theory and chemistry \cite{Bapat9, Chemistry1, Bollobas,Chemistry2, Klein,Chemistry3, Chemistry4, Stevenson},
combinatorial matrix theory \cite{Bapat0,Yang}, and
spectral graph theory \cite{Bapat9,Bollobas,Chenzhang,Spielman}.
 
A recent paper 
 \cite{Barrett9}
 devoted to 2-linear trees,   mentions, in passing, 
the triangular grid and offers a conjecture on  the total resistance between any two corner vertices (vertices of degree 2). The proofs in \cite{Barrett9}  utilize a \textit{row-reduction} algorithm   which uses  transformations that preserve electric resistance to reduce an $n$-row triangular grid to an $n-1$ triangular grid. 

The current paper explores computational uses of this row reduction algorithm. In the introductory part of the paper several conjectures are presented, supported by  numerical evidence, showing that repeated application of the row-reduction algorithm to an initial $n$-grid uniformly labeled 1 asymptotically produces  triangular grids whose sides are labeled with rational multiples of $\frac{1}{e};$ moreover, the ratio of specified consecutive edges in these grids are asymptotically equal to the values of four rational functions.

Motivated by this conjecture, the  main part of this paper studies a family of graphs whose edge labels are determined  using the limiting edge-ratio rational functions arising in the conjectures.   These graphs have interest independent of the conjectures which motivated them since the graphs of circuits in the literature typically are uniformly labeled one.  

The main result of this   paper is that the resulting $n$-grids including their repeated reductions under the row-reduction algorithm   possess vertical and rotational symmetries and satisfy the relationships implied by the four limiting edge-ratio rational functions; in other words,   these four limiting edge-ratio rational functions  are local algebraic relationships   that mirror the global vertical and rotational symmetries possessed by the underlying graph

In the next section (Section \ref{sec:algorithm}) we introduce notations and conventions used throughout the paper as well as present the row-reduction algorithm. This is followed by the precise formulation of, and the numerical evidence for,  the seven conjectures (Section \ref{sec:conjectures}). The  rest of the paper motivates (Section \ref{sec:motivatingexample}) and states the three main theorems of the paper (Section \ref{sec:maintheorems}) which are then proven in the concluding sections    (Sections \ref{sec:isotropy}-\ref{sec:completion}).

\section{Notations, Conventions, and Algorithms}\label{sec:algorithm} This section presents definitions, notations, and conventions for five items: the $n$-grid, a review of resistance-preserving electric-circuit transformations, the row-reduction algorithm, the four local edge functions, and proof methods. The definitions and notations are found in \cite{EvansHendel} and therefore attribution will not be repeated on each item.

  Definition \ref{def:ngrid} and Figure \ref{fig:3grid} present the definition of the $n$-grid and illustrate conventions and notations associated with it.

\begin{definition}\label{def:ngrid}
A $n$-triangular-grid (usually abbreviated as an $n$-grid)
is any graph that is (graph-) isomorphic to the graph,
whose vertices are all integer pairs $(x,y)=(2r+s,s)$ 
in the Cartesian plane, with $r$ and $s$ integer
parameters satisfying
$0 \le r \le n, 0 \le s \le n-r;$
  and whose edges
consist of any two vertices $(x,y)$ and $(x',y')$ with
(i) $x'-x=1, y'-y=1,$ 
 (ii) $x'-x=2, y'-y=0,$ or
(iii) $x'-x=1, y'-y=-1.$   
\end{definition}

\begin{figure}[ht!]
\begin{center}

\caption{A 3-grid embedded in the Cartesian Plane constructed using Definition~\ref{def:ngrid} (Left Panel) with the paper's conventions on row and diagonal coordinates (Right Panel).}\label{fig:3grid}
\begin{tabular}{|c|}
\hline
\begin{tikzpicture} [xscale=1,yscale =1]

\node [below] at (3,3.2) {    }; 
 
\draw (0,0)--(1,1)--(2,2)--(3,3);
\draw (2,0)--(3,1)--(4,2);
\draw (4,0)--(5,1);

\draw (3,3)--(4,2)--(5,1)--(6,0);
\draw (2,2)--(3,1)--(4,0);
\draw (1,1)--(2,0);

\draw (0,0)--(2,0)--(4,0)--(6,0);
\draw (1,1)--(3,1)--(5,1);

\draw (2,2) -- (4,2);

\begin{scriptsize}
\node [below] at (3,2.8) {$(3,3)$};

\node [below] at (2,1.8) {$(2,2)$};
\node [below] at (4,1.8) {$(4,2)$};

\node [below] at (1,.8) {$(1,1)$};
\node [below] at (3,.8) {$(3,1)$};
\node [below] at (5,.8) {$(5,1)$};

\node [below] at (6,0) {$(6,0)$};
\node [below] at (4,0) {$(4,0)$};
\node [below] at (2,0) {$(2,0)$};
\node [below] at (0,0) {$(0,0)$};

\end{scriptsize}
\end{tikzpicture}

\begin{tikzpicture} [xscale=1,yscale =1]

\node [below] at (3,3.2) {    }; 
 
\draw (0,0)--(1,1)--(2,2)--(3,3);
\draw (2,0)--(3,1)--(4,2);
\draw (4,0)--(5,1);

\draw (3,3)--(4,2)--(5,1)--(6,0);
\draw (2,2)--(3,1)--(4,0);
\draw (1,1)--(2,0);

\draw (0,0)--(2,0)--(4,0)--(6,0);
\draw (1,1)--(3,1)--(5,1);

\draw (2,2) -- (4,2);

\begin{scriptsize}
\node  [above] at (1,.2) {$T_{3,1}$};  
\node  [above] at (3,.2) {$T_{3,2}$};
\node  [above] at (5,.2) {$T_{3,3}$};
\node  [above] at (2,1.2) {$T_{2,1}$};
\node  [above] at (4,1.2) {$T_{2,2}$};
\node  [above] at (3,2.3) {$T_{1,1}$};

\node [below] at (6,0) {$ $};
\node [below] at (4,0) {$ $};
\node [below] at (2,0) {$ $};
\node [below] at (0,0) {$ $};

\end{scriptsize}

\end{tikzpicture}
\\  \hline
\end{tabular}

\end{center}
\end{figure}

 Recall that electric circuit theory introduces four transformations - series, parallel, $\Delta$-Y, and Y-$\Delta$ that preserve (electric) resistance. The series transformation, for example, takes a labeled 3-vertex subgraph with two edges and transforms (or replaces) it with  a 2-vertex subgraph whose single edge is labeled with the sum of the labels (resistances) of the two edges of the 3-vertex subgraph. Thus, this transformation preserves resistance.
  Equation \ref{equ:DeltaY} and Figure  \ref{fig:DeltaY} review the $\Delta$-Y and Y-$\Delta$ circuit transformations which transform initial circuits into simpler circuits with equivalent electric resistances.  The figure establishes several conventions about functional notation and labels used throughout the paper.
 \begin{equation}\label{equ:DeltaY}
 \Delta(x,y,z) = \frac{xy}{x+y+z}, \qquad
 Y(a,b,c) = \frac{ab+bc+ca}{a}.
 \end{equation}

 Throughout the paper, we refer to the edges of the triangles using the notation $T_{X,r,d}$ for some $X \in \{L, R, B\}$ which respectively stand for the left, right, and base side of the triangle. Note that we abuse notation; for example,   $L$ refers both to the left edge and to the label of the left edge; however,  meaning will always be clear from context.  The notation for the Y legs is similarly constructed: e.g.   $Y_{12,1,2}$ is the 12 o'clock leg arising from applying a Y-$\Delta$ transformation to $T_{1,2}.$ 
 
 
 \begin{figure}[ht!]
\begin{center}\caption{Definition and orientation conventions for the 
 $\Delta$-Y and Y-$\Delta$ circuit transformations.   $S, T, U$ are vertex labels; $L, R, B$ are edge labels standing for left, right, and base edge. The $Y$ legs are labeled with an intuitive clock notation. }\label{fig:DeltaY}
\begin{tabular} {|c|}

\hline
\begin{tikzpicture}[yscale=.4,xscale=1]

\draw (0,0)--(2,4)--(4,0)--(0,0) ;
\draw (5,0)--(7,2)--(7,4)--(7,2)--(9,0);

\node [below] at (7,4.2) { };

\node [above] at (7,4) {T};
\node [left] at (5,0) {S};
\node [right] at (9,0) {U};

\node [above] at (2,4) {T};
\node [left] at (0,0) {S};
\node [right] at (4,0) {U};

\node [right] at (7,3) {$Y_{12}$};
\node [right] at (8.1,1.1) {$Y_{4}$};
\node [right] at (5.6,.3) {$Y_{8}$}; 

\node [right] at (3,2.2) {$R$};
\node [right] at (2,-.6) {$B$};
\node [right] at (.5,2) {$L$}; 

\node [right] at (6.5,-2) {$Y_{12}=\Delta(L,R,B)$};
\node [right] at (6.5,-3) {$Y_{4}=\Delta(R,B,L)$};
\node [right] at (6.5,-4) {$Y_{8}=\Delta(B,L,R)$}; 

\node [right] at (.75, -2) {$R=Y(Y_8,Y_{12},Y_4)$ };
\node [right] at (.75,-3) {$B=Y(Y_{12},Y_{4},Y_8)$ };
\node [right] at (.75,-4) {$L=Y(Y_4,Y_8,Y_{12})$ }; 
\end{tikzpicture}
\\ \hline
\end{tabular}
\end{center}
\end{figure}

 Equation \eqref{equ:DeltaY} implies the following scaling properties of these functions.
\begin{equation}\label{equ:scaling} \Delta(jx, jy, jz) = j \Delta(x,y,z), \qquad    
    Y(jx,jy,jz) = j Y(x,y,z), \text{ for any real $j$}.
\end{equation}

 The reduction algorithm presented next is based on proofs in \cite[pg. 18]{Barrett9} connected with linear 2-trees and 
 accompanied by figures which are adapted and expanded here for application to the $n$-grid. The algorithm presented in Definition \ref{def:rr} has five steps which, in the sequel, will be referred to as Steps A-E. These steps are illustrated in the three rows of Figure \ref{fig:12panels}.

\begin{definition}\label{def:rr}   Given a (parent) $n$-grid with labeled edges, the   
\emph{row-reduction} of this labeled $n$-grid (to a labeled (child) $n-1$-grid) 
refers to the sequential performance of the following steps, the steps being illustrated   by Figure \ref{fig:12panels}.
\begin{itemize}
    \item Step A: Start with an \textit{$n$-grid} (Figure \ref{fig:12panels} illustrates with $n=3 \text{ and } 2.$). 
    \item Step B: Apply a $\Delta$--Y transformation to each upright triangle (a 3-loop) resulting in a grid of $n$ rows of out-stars.
    \item Step C: Ignore the corner tails,  edges dotted in Step B with a vertex of degree one, since they do not affect the resistance labels of edges in the reduced two grid shown in Step 5. (However, these corner tails are useful for computing certain effective resistances as shown in 
    \cite{Barrett0,Evans2022}).
    \item Step D: Perform series transformations on all consecutive pairs of boundary edges (i.e., the dashed edges in Step C).
    \item Step E: Apply Y--$\Delta$ transformations to any remaining out-stars, transforming them into loops.
\end{itemize}   
 \end{definition}

Since Steps B-E involve functions that produce equivalent electric resistances, the resistance between two corner vertices in the resulting $n-1$ grid is equal to the resistance between two corner vertices in the original $n$ grid.

 \begin{figure}[ht!]
\begin{center}
\caption{Illustration of two row reductions, Definition \ref{def:rr}, on  a 3-grid whose edges are uniformly labeled with 1.}\label{fig:12panels}
\begin{tabular}{||c|c|c|c|c||}
\hline
\begin{tikzpicture}[yscale=.4,xscale=.3]

\node [below] at (3,3.2) {    }; 
 
\draw (0,0)--(1,2)--(2,0)--(0,0);
\draw (2,0)--(3,2)--(4,0)--(2,0);
\draw (4,0)--(5,2)--(6,0)--(4,0); 
\draw (1,2)--(2,4)--(3,2)--(1,2);
\draw (3,2)--(4,4)--(5,2)--(3,2);
\draw (2,4)--(3,6)--(4,4)--(2,4);

\node [below] at (3,6.2) {     };
\node [below] at (3,-.5) {Step A1};
\end{tikzpicture} 
& 
\begin{tikzpicture}[xscale=.3,yscale=.4] 
\draw  (1,1)--(1,2)--(2,3)--(2,4)--(3,5);
\draw  (5,1)--(5,2)--(4,3)--(4,4)--(3,5);
\draw (1,1)--(2,0)--(3,1)--(3,2)--(2,3);
\draw (3,1)--(4,0)--(5,1)--(5,2)--(4,3)--(3,2); 
 
\draw [dashed] (6,0)--(5,1);
\draw [dashed] (0,0)--(1,1);
 \draw [dashed] (3,5)--(3,6);
 
\node [below] at (3,-.5) {Step B1};
\end{tikzpicture}&
\begin{tikzpicture}[yscale=.4,xscale=.3]
 
\draw [dashed] (1,1)--(1,2)--(2,3)--
(2,4)--(3,5);
\draw [dashed] (5,1)--(5,2)--(4,3)--(4,4)--(3,5);
\draw [dashed] (1,1)--(2,0)--(3,1)--(4,0)--(5,1);
\draw (3,1)--(3,2)--(2,3)--(3,2)--(4,3);

\node [below] at (3,-.5) {Step C1};
 
\end{tikzpicture}&
\begin{tikzpicture}[yscale=.4,xscale=.3]
 
\draw (1,1)--(2,3)--(3,5);
\draw (5,1)--(4,3)--(3,5);
\draw (1,1)--(3,1)--(5,1);
\draw [dotted] (2,3)--(3,2)--(4,3);
\draw [dotted] (3,2)--(3,1); 
 
 \node [below] at (3,-.5) {Step D1};
\end{tikzpicture}&
\begin{tikzpicture}[yscale=.4,xscale=.3]
 
\draw (1,1)--(2,3)--(3,5);
\draw (5,1)--(4,3)--(3,5);
\draw (1,1)--(3,1)--(5,1);
\draw (2,3)--(4,3)--(3,1)--(2,3); 
 
 \node [below] at (3,-.5) {Step E1};
\end{tikzpicture} 
\\

\hline

\begin{tikzpicture}[yscale=.4,xscale=.3]

\draw [dotted] (0,0)--(1,2)--(2,0)--(0,0);
\draw  [dotted] (2,0)--(3,2)--(4,0)--(2,0);
\draw  [dotted] (4,0)--(5,2)--(6,0)--(4,0);
\draw  [dotted] (1,2)--(2,4)--(3,2)--(1,2);
\draw  [dotted] (3,2)--(4,4)--(5,2)--(3,2);
\draw  [dotted]  (2,4)--(3,6)--(4,4)--(2,4);

\node [below] at (3,0) {Step A2};
\node [above] at (3,6) {All edge labels are 1};

\end{tikzpicture}&

\begin{tikzpicture}[yscale=.4,xscale=.3]
 
\draw [dotted] (0,0)--(1,1)--(1,2)--(1,1)--(2,0);
\draw [dotted] (2,0)--(3,1)--(3,2)--(3,1)--(4,0);
\draw [dotted] (4,0)--(5,1)--(5,2)--(5,1)--(6,0);
\draw [dotted] (1,2)--(2,3)--(2,4)--(2,3)--(3,2);
\draw [dotted] (3,2)--(4,3)--(4,4)--(4,3)--(5,2);
\draw [dotted] (2,4)--(3,5)--(3,6)--(3,5)--(4,4);
 
\node [below] at (3,-.5) {Step B2};
\node [above] at (3,6) {All edge labels are $\frac{1}{3}$};

\end{tikzpicture}&

\begin{tikzpicture}[yscale=.4,xscale=.3]

\draw [dotted]        (1,1)--(1,2)--(1,1)--(2,0);
\draw [dotted] (2,0)--(3,1)--(3,2)--(3,1)--(4,0);
\draw [dotted] (4,0)--(5,1)--(5,2)--(5,1)       ;
\draw [dotted] (1,2)--(2,3)--(2,4)--(2,3)--(3,2);
\draw [dotted] (3,2)--(4,3)--(4,4)--(4,3)--(5,2);
\draw [dotted] (2,4)--(3,5)--(4,4);

 \node [below] at (3,-.5) {Step C2};
\node [above] at (3,6) {All edge labels are $\frac{1}{3}$};
 
\end{tikzpicture}&

\begin{tikzpicture}[yscale=.4,xscale=.3]

\draw [dotted] 	(1,1)--(2,3)--(3,5)--(4,3)--(5,1)--(1,1);
\draw [dotted] (3,1)--(3,2)--(2,3);
\draw [dotted] (3,2)--(4,3);

\node [left] at (1.5,2) {$\frac{2}{3}$};
\node [left] at (2.5,4) {$\frac{2}{3}$};
\node [right] at (3.5,4) {$\frac{2}{3}$};
\node [right] at (4.5,2) {$\frac{2}{3}$};
\node [below] at (2,1) {$\frac{2}{3}$};
\node [below] at (4,1) {$\frac{2}{3}$};

\node [right] at (3,1.5) {$\frac{1}{3}$};
\node [right] at (2,2.5) {$\frac{1}{3}$};
\node [right] at (3,2.5) {$\frac{1}{3}$};

\node [below] at (3,-.5) {Step D2};
\node [above] at (3,6) {         };

\end{tikzpicture}&

\begin{tikzpicture}[yscale=.4,xscale=.3] *
\draw [dotted] (1,1)--(2,3)--(3,1)--(1,1);
\draw [dotted] (2,3)--(3,5)--(4,3)--(2,3);
\draw [dotted] (3,1)--(4,3)--(5,1)--(3,1);
\draw [dotted] (2,3)--(4,3)--(3,1)--(2,3);

 \node [above] at (3,3) {$1$};
\node [left] at (2.5,2) {$1$};
\node [right] at (3.5,2) {$1$};

\node [left] at (1.5,2) {$\frac{2}{3}$};
\node [left] at (2.5,4) {$\frac{2}{3}$};
\node [right] at (3.5,4) {$\frac{2}{3}$};
\node [right] at (4.5,2) {$\frac{2}{3}$};
\node [below] at (2,1) {$\frac{2}{3}$};
\node [below] at (4,1) {$\frac{2}{3}$};

\node [below] at (3,-.5) {Step E2};
\node [above] at (3,6) { };
\end{tikzpicture} 

\\ \hline

\begin{tikzpicture}[yscale=.4,xscale=.3]

\draw [dotted] (1,1)--(2,3)--(3,1)--(1,1);
\draw [dotted] (2,3)--(3,5)--(4,3)--(2,3);
\draw [dotted] (3,1)--(4,3)--(5,1)--(3,1);
\draw [dotted] (2,3)--(4,3)--(3,1)--(2,3);

\node [above] at (3,3) {$1$};
\node [left] at (2.5,2) {$1$};
\node [right] at (3.5,2) {$1$};

\node [left] at (1.5,2) {$\frac{2}{3}$};
\node [left] at (2.5,4) {$\frac{2}{3}$};
\node [right] at (3.5,4) {$\frac{2}{3}$};
\node [right] at (4.5,2) {$\frac{2}{3}$};
\node [below] at (2,1) {$\frac{2}{3}$};
\node [below] at (4,1) {$\frac{2}{3}$}; 

\node [below] at (3,0) {Step A3};
\node [above] at (3,4.2) {   };

\end{tikzpicture}&

\begin{tikzpicture}[yscale=.4,xscale=.3]

\draw [dotted] (1,1)--(2,2)--(2,3)--(2,2)--(3,1);
\draw [dotted] (2,3)--(3,4)--(3,5)--(3,4)--(4,3);
\draw [dotted] (3,1)--(4,2)--(4,3)--(4,2)--(5,1);

\node [left] at (1.5,1.5) {$\frac{4}{21}$};
\node [left] at (2,2.5) {$\frac{2}{7}$};
\node [left] at (2.5,3.5) {$\frac{2}{7}$};
\node [left] at (3,4.5) {$\frac{4}{21}$};
\node [right] at (4.5,1.5) {$\frac{4}{21}$};

\draw [dotted]        (2,2)--(2,3)--(2,2)--(3,1);
\draw [dotted] (2,3)--(3,4)--(4,3);
\draw [dotted] (3,1)--(4,2)--(4,3)--(4,2)        ;

 \node [left] at (2,2.5) {$\frac{2}{7}$};
\node [left] at (2.5,3.5) {$\frac{2}{7}$};

\node [right] at (3.5,3.5) {$\frac{2}{7}$};
\node [right] at (4,2.5) {$\frac{2}{7}$};
\node [right] at (2,1.5) {$\frac{2}{7}$};
\node [right] at (3,1.5) {$\frac{2}{7}$};

\node [below] at (3,1) {Step B3};
\node [above] at (3,4.2) {     };

\end{tikzpicture}&

\begin{tikzpicture}[yscale=.4,xscale=.3]

\draw [dotted]        (2,2)--(2,3)--(2,2)--(3,1);
\draw [dotted] (2,3)--(3,4)--(4,3);
\draw [dotted] (3,1)--(4,2)--(4,3)--(4,2)        ;

 \node [left] at (2,2.5) {$\frac{2}{7}$};
\node [left] at (2.5,3.5) {$\frac{2}{7}$};

\node [right] at (3.5,3.5) {$\frac{2}{7}$};
\node [right] at (4,2.5) {$\frac{2}{7}$};
\node [right] at (2,1.5) {$\frac{2}{7}$};
\node [right] at (3,1.5) {$\frac{2}{7}$};

\node [below] at (3,-.5) {Step C3};
\node [above] at (3,4.2) { };

\end{tikzpicture}&

\begin{tikzpicture}[yscale=.4,xscale=.3]

\draw [dotted] (2,2)--(3,4)--(4,2)--(2,2);

\node [left] at (2.5,3) {$\frac{4}{7}$};
\node [right] at (3.5,3) {$\frac{4}{7}$};
\node [below] at (3,2) {$\frac{4}{7}$};
 
\node [below] at (3,-.5) {Step D3};
\node [above] at (3,4.2) {  };
\end{tikzpicture}
&

\\ \hline
 
\end{tabular}
\end{center}
 \end{figure}

 Figure \ref{fig:12panels} also   illustrates \textit{repeated row-reduction}. The initial 3 grid uniformly labeled 1,  transforms first to a 2-grid and then a 1-grid. The edge-labels throughout Figure \ref{fig:12panels} are justified by \eqref{equ:DeltaY},  for example,
 $\frac{4}{21} = \Delta(\frac{2}{3},\frac{2}{3},1),$  
  and similarly, 
  $\frac{2}{7} = \Delta(1, \frac{2}{3},\frac{2}{3}).$  

  The sequence of steps in the algorithm naturally generate several single functions whose arguments are edge-label values in the parent grid and whose return-value is an edge-label in the child grid.  To illustrate, for $1 \le r \le n-1,$ define a function $\mathcal{P}$ which returns the value of $T_{L,r,1}$ in a (child) reduced grid as a function of six edge values in the original parent grid. ($\mathcal{P}$ stands for \textit{perimeter} and helps distinguish between the function for the left edge on the perimeter of the grid and the non-boundary left edges).

\begin{equation}\label{equ:perimeteredge}
T_{L,r,1}= \mathcal{P}(T_{r,1},T_{r+1,1}) =Y_{8,r,1}+Y_{12,r+1,1}=
\Delta(T_{B,r,1},T_{L,r,1},T_{R,r,1})+
\Delta(T_{L,r+1,1},T_{R,r+1,1},T_{B,r+1,1}).
\end{equation}

While the notation can be made more precise using supersripts to indicate the parent and once-reduced grid, we simplify notation by using the  convention that    equations similar to  \eqref{equ:perimeteredge} are interpreted so that triangle-edge values on the  left-hand side refer to labels in the the child (that is, the once-reduced $n$-) grid while those on the right-hand side refer to labels in the the parent grid. Moreover, for readability, we have also indicated the arguments of $\mathcal{R}$ using triangles.  Note that, for example,  if we let $r=1$ in \eqref{equ:perimeteredge} we obtain an equation justifying the $\frac{4}{7}$ label in Step D3. 

 Only four local-edge functions are needed for all computations. The other three local edge functions, $  \mathcal{L}, \mathcal{R},$ and $\mathcal{B}$ for the labels of a non-boundary left edge, right-edge, and base edge, respectively, with the convention just indicated that the left-hand and right-hand side refer to items in the child (reduced) and parent grids, are as follows.

For a given $n$-grid for $1 \le r\le n-2, 1 \le d \le r-1,$ the function $\mathcal{R}$ gives the label of a base-edge in a once-reduced $n$ grid. 
\begin{multline}\label{equ:baseedge}
\mathcal{B}(T_{r+2,d+1} 
            T_{r+1,d},
            T_{r+1,d+1}) = 
T_{r,d,B} =
    Y(\Delta(T_{L,r+2,d+1},
            T_{R,r+2,d+1},
            T_{B,r+2,d+1}),
            \Delta(T_{R,r+1,d},
            T_{B, r+1,d},
            T_{L,r+1,d}),\\
             \Delta(T_{B, r+1,d+1},
        T_{L, r+1,d+1},
            T_{R, r+1,d+1})).
\end{multline}
.

For an $n$-grid for 
$1 \le r \le n-2, 1 \le d \le r-1$
the function $\mathcal{R}$ gives 
the label of a right-edge in a once-reduced $n$ grid. 
\begin{multline}\label{equ:rightedge}
\mathcal{R}(T_{r,d} 
            T_{r,d+1},
            T_{r+1,d})
=
T_{r,d,R} =
    Y(\Delta(T_{r,d,R},
            T_{r,d,B},
            T_{r,d,L}),
      \Delta(T_{r,d+1,B},
        T_{r,d+1,L},
            T_{r,d+1,R},\\      
      \Delta(T_{r+1,d,L},
            T_{r+1,d,R},
            T_{r+1,d,B})).
\end{multline}

For a given $n$-grid with  
$
    2 \le r \le n-1, 2 \le d \le r.
$
the function $\mathcal{L}$ gives the label of a non-boundary left edge in a one reduced $n$-grid.
 \begin{multline}\label{equ:leftedge}
\mathcal{L}(T_{r,d-1} 
            T_{r,d},
            T_{r+1,d})
T_{r,d,L} =
    Y(\Delta(T_{R,r,d-1},
            T_{B,r,d-1},
            T_{L,r,d-1}),
     \Delta(T_{B,r,d},
        	T_{L,r,d},
            T_{R,r,d}),\\  
            \Delta(T_{L,r+1,d},
            T_{R,r+1,d},
            T_{B,r+1,d})).
        \end{multline}

In the sequel we will refer to \eqref{equ:perimeteredge}-\eqref{equ:leftedge} as the \textit{the four local functions.} As illustrated with $\mathcal{P},$ the proof that these functions accomplish their stated goals is provided by Definition \ref{def:rr}.

The four local functions justify the following approach to proofs used throughout the paper.
Traditionally, the combinatoric Laplacian is used to justify computational identities. But the combinatoric Laplacian is a global method involving the entire graph while the four local functions
are computed locally (with at most 9 edge-labels as function arguments). Consequently, if a proof reduces to verification that two (rational) functions are equal, we will suffice with simply stating \textit{we may verify,} the verification being accomplishable either manually or by algebraic software. Although some proofs can easily be done without software, the narrative flows more naturally if the computational details of the proofs are always left to verification.

\section{The Conjectures}\label{sec:conjectures}

  Figure \ref{fig:12panels} can be used to  motivate the seven conjectures presented in this paper.

\begin{example}\label{exa:tailconjecture} Figure \ref{fig:12panels} starts in Steps A1 and A2 with a  3-grid  uniformly labeled 1,  and then applies two repeated row reductions resulting in the 1-grid of Step D3 with edge label $\frac{4}{7}.$ But,
$$
    \frac{4}{7}=0.5714  \approxeq 0.5518 = \frac{3}{2} \frac{1}{e}, 
$$
with an error of 3.55\%. 

Continuing this example, the ignored tail  in Step B3 has   label $\frac{4}{21}.$ We have  
$$
    \frac{4}{21} =0.1905 \approxeq  0.1839 = \frac{1}{2}\frac{1}{e}.
$$
\end{example}

Conjecture 1 will be formulated in terms of tails rather than edges. The reason for this is that  in a 1-grid arising from repeated row-reduction, there is a nice relationship between the resistance distance of the edges of the 1 grid and the resistance distance of the ignored tail arising in the process of row reducing a 2-grid to a 1-grid. But in general, the edge labels of $c$-grids, $c >1,$ (arising from applying $n-c$ repeated row reductions to an initial $n$-grid uniformly labeled 1) do not have a nice relationship with tails arising in the row reduction process.  To state the conjectures, we therefore need to create a notation for the resistance distance of tails. 

\begin{definition}
    \label{def:tail} $t(n,i)$ indicates the label of any tail in Step 2 of the $n-i$-th row reduction of an original $n$-grid. \end{definition}

\begin{example}\label{exa:tail} In Figure \ref{fig:12panels}, $t(3,2)=\frac{4}{21}.$ Notice that $t(3,2)$ depends on the labels in the initial $n$-grid.
\end{example} 

Conjecture \ref{con:con1} generalizes and formalizes Example \ref{exa:tailconjecture}.  
 
\begin{con}\label{con:con1}  Consider a family of initial $n$-grids whose edges are uniformly labeled one. As $n$ goes to infinity,
$$
        t(n,1) \approxeq \frac{1}{2e}.
$$
\end{con}

\begin{con}\label{con:con2}  Fix $i \ge 1.$ With the conventions of Conjecture \ref{con:con1},  as $n$ goes to infinity,
$$
        t(n,i) \approxeq \frac{1}{i} \frac{1}{2e}.
$$
\end{con}

 Table \ref{tab:con1} provides supportive numerical evidence for both conjectures.

\begin{table}[ht!]
\caption{Numerical evidence for Conjectures \ref{con:con1} and 
\ref{con:con2}. }\label{tab:con1}.
\begin{tabular}{||c||c|c|c|c|c|c|c|c|c||} 
\hline \hline

$i=$&
$Actual\; value\; of\; t(150,i)$&
$Conjectured \; value$&
$Numerical\; Error\; (Ratio-1) $\\
\hline
$1$&$0.183776286$&$\frac{1}{2e} = 0.183939721$&$0.0009$\\
$2$&$0.091888053$&$\frac{1}{2} \frac{1}{2e} = 0.09196986$&$0.0009$\\
$3$&$0.061258443$&$\frac{1}{3} \frac{1}{2e} = 0.06131324$&$0.0009$\\
$4$&$0.045943311$&$\frac{1}{4} \frac{1}{2e} = 0.04598493$&$0.0009$\\
$5$&$0.036753773$&$\frac{1}{5} \frac{1}{2e} = 0.036787944$&$0.0009$\\
$6$&$0.030626823$&$\frac{1}{6} \frac{1}{2e} = 0.03065662$&$0.001$\\
$7$&$0.026249708$&$\frac{1}{7} \frac{1}{2e} = 0.026277103$&$0.001$\\
$8$&$0.02296602$&$\frac{1}{8} \frac{1}{2e} = 0.022992465$&$0.0012$\\
$9$&$0.020411064$&$\frac{1}{9} \frac{1}{2e} = 0.020437747$&$0.0013$\\
$10$&$0.018366$&$\frac{1}{10} \frac{1}{2e} = 0.018393972$&$0.0015$\\

\hline \hline
\end{tabular}

\end{table}
The next four conjectures study patterns in the ratios of certain edge-labels in the $c$ grid resulting from applying $n-c$ row reductions to an $n$ grid whose edges are uniformly labeled 1. The conjectures should be read as defining  four functions, $r_{R,L}, r_{B,L}, x,y$   and then conjecturing that the ratio of the labels specified in each conjecture have the indicated value of these functions. In the sequel, we will refer to these four functions as \textit{the (four) edge-ratio functions}.
 
\begin{con}\label{con:con3}
Fix $c \ge 1.$ With the conventions of Conjecture \ref{con:con1}, as $n$ goes to infinity,
  \begin{equation}\label{equ:edgefactorsright}
	\frac{T_{R,r,d}}{T_{L,r,d}} \approxeq \frac{2(r-d)+1}{2d-1} \doteq r_{R,L}(r,d),
	\qquad   1 \le d \le r \le c.    
\end{equation}
\end{con} 

\begin{con}\label{con:con4}
Fix $c \ge 1.$ With the conventions of Conjecture \ref{con:con1}, as $n$ goes to infinity,
\begin{equation}\label{equ:edgefactorsbase} 
	\frac{T_{B,r,d}}{T_{L,r,d}} \approxeq 
	\frac{2(c-r)+1}{2d-1} \doteq r_{B,L}(c,r,d), \qquad
	  1 \le d \le r \le c.    
\end{equation}
\end{con} 

\begin{con}\label{con:con5}
Fix $c \ge 1.$ With the conventions of Conjecture \ref{con:con1}, as $n$ goes to infinity,
\begin{equation}\label{equ:edgefactorsleft}
	 \frac{T_{L,r,1}}{T_{L,r-1,1}} \approxeq
	\frac{r-1}{2r-1}\frac{2(c-r)+3}{(c-r)+1} \doteq x(c,r), \qquad
	  2 \le r \le c.  				 
\end{equation}
\end{con} 

\begin{con}\label{con:con6}
Fix $c \ge 1.$ With the conventions of Conjecture \ref{con:con1}, as $n$ goes to infinity,
\begin{equation}\label{equ:edgefactorshorizontal}
	\frac{T_{L,r,d}}{T_{L,r,d-1}} \approxeq 
	\frac{d-1}{2d-3} \frac{2(r-d)+3}{(r-d)+1} \doteq y(r,d), \qquad
	2 \le d \le r \le c.   
\end{equation}
\end{con}

Table \ref{tab:con2} presents 
 numerical evidence for Conjectures \ref{con:con3}-\ref{con:con6}.

\begin{table}[ht!]
\caption
{Numerical evidence for Conjectures \ref{con:con3}-\ref{con:con6} based on 140 consecutive reductions of an initial 150-grid uniformly labeled 1.}\label{tab:con2}
\begin{tabular}{||c|c|c|c|c||} 
\hline \hline
\text{Edge ratio}&\text{Actual value of edge ratios }&\text{Equation reference}&\text{Predicted edge ratio Value}&\text{Error}\\
\hline
$\frac{T_{R,4,1}}  {T_{L,4,1}}$ \footnotemark[1]&$7.0010150391$&$\eqref{equ:edgefactorsright}$&$r_{R,L}(4,1)=7$&$-0.00014$\\
$\frac{T_{R,4,2}}{T_{L,4,2}}$&$1.6667476811$&$\eqref{equ:edgefactorsright}$&$r_{R,L}(4,2)=\frac{5}{3}$&$-0.00005$\\
$\frac{T_{R,6,4}}{T_{L,6,4}}$&$0.7142421068$&$\eqref{equ:edgefactorsright}$&$r_{R,L}(6,4)=\frac{5}{7}$&$-0.00006$\\
\;&\;&\;&\;&\;\\
$\frac{T_{B,1,1}}{T_{L,1,1}}$&$19.0125047453$&$\eqref{equ:edgefactorsbase}$&$r_{B,L}(10,1,1)=19$&$-0.00066$\\
$\frac{T_{B,5,3}}{T_{L,5,3}}$&$2.20046025$&$\eqref{equ:edgefactorsbase}$&$r_{B,L}(10,5,3)=\frac{11}{5}$&$-0.00021$\\
$\frac{T_{B,6,4}}{T_{L,6,4}}$&$1.2858046383$&$\eqref{equ:edgefactorsbase}$&$r_{B,L}(10,6,4)=\frac{9}{7}$&$-0.00007$\\
\;&\;&\;&\;&\;\\
$\frac{T_{L,2,1}}{T_{L,1,1}}$&$0.7036757582$&$\eqref{equ:edgefactorsleft}$&$x(10,2)=\frac{1}{3}\frac{19}{9}$&$0.00004$\\
$\frac{T_{L,3,1}}{T_{L,2,1}}$&$0.8499699096$&$\eqref{equ:edgefactorsleft}$&$x(10,3)=\frac{2}{5}\frac{17}{8}$&$0.00004$\\
$\frac{T_{L,4,1}}{T_{L,3,1}}$&$0.9183432041$&$\eqref{equ:edgefactorsleft}$&$x(10,4)=\frac{3}{7}\frac{15}{7}$&$0.00003$\\
\;&\;&\;&\;&\;\\
$\frac{T_{L,4,2}}{T_{L,4,1}}$&$2.3334191779$&$\eqref{equ:edgefactorshorizontal}$&$y(4,2)=\frac{7}{3}$&$-0.00004$\\
$\frac{T_{L,4,3}}{T_{L,4,2}}$&$1.6667476811$&$\eqref{equ:edgefactorshorizontal}$&$y(4,3)=\frac{5}{3}$&$-0.00005$\\
$\frac{T_{L,5,4}}{T_{L,5,3}}$&$1.500093071$&$\eqref{equ:edgefactorshorizontal}$&$y(5,4)=\frac{3}{2}$&$-0.00006$  \\

\hline \hline
\end{tabular}
\footnotetext[1]{The interpretation of the first table row   states that the quotient of the labels of 
the right  and left edge of triangle $T_{4,1}$
in the 10 grid arising from applying 140 row reductions to an initial 150-grid uniformly labeled 1 is $7.001 \dotsc.$  However,  \eqref{equ:edgefactorsright}   predicts a ratio of $r_{R,L}(10,4,1)=7,$ implying a very small error.}
%
\end{table}

  As mentioned earlier, \cite{Barrett9}, in introducing the triangular grid also presented  a conjecture \cite[Conjecture 7.8]{Barrett9} that if $r_n$ denotes the resistance distance between any two corners (degree two vertices) of the $n$-grid, then as $n$ approaches infinity  $e^{r_{n+1}}-e^{r_n} =C,$ for some constant $C.$

It is straightforward to show that the resistance between any two degree-2 nodes in an  $n$-grid is exactly equal to
 $2 \times \sum_{i=1}^n t(n,i).$ 
      But by Conjecture \ref{con:con1}, $t(n,i) \approxeq \frac{1}{i} t(n,1)$ provided $n$ is large enough and $i$ is independent of $n$. Motivated by these considerations, Evans and I were able to refine the conjectured asymptotic limit stated in \cite{Barrett9}.
 
 \begin{con}\label{con:con7} 
 $$\text{The resistance distance between two corner nodes of the $n$-grid} \approxeq
 \sum_{i=1}^n \frac{1}{i}.$$ 
\end{con}.

None of these conjectures will be proven in this paper. Athough they are interesting, they are presented because they   naturally suggest other  avenues of exploration. In  the remainder of this paper we study defining the labels of an initial $n$-grid using the four edge-ratio functions. This contrasts with most examples in the literature which almost always present circuits  uniformly labeled with ones. 

This study will result in three main theorems which collectively state that both the initial $n$-grid defined by the four edge-ratio functions as well as all $n-i$ grids arising from $i, 1 \le i \le n-1,$ repeated row reductions of the initial $n$ grid possess vertical and rotational symmetry and satisfy the relationships implied by the four edge-ratio functions.  The tails arising in the row reductions manifest behavior similar to that stated in Conjecture \ref{con:con2}.

These theorems shed light on the four edge-ratio functions. They are local algebraic relationships which correspond to the global symmetries present in the initial and reduced grids.  
\section{A Motivating Example}\label{sec:motivatingexample}

This section motivates the idea of an initial $c$-grid whose edges are not uniformly labeled 1. The section provides definitions, a worked out example, and notes various patterns which will be formalized in the three main theorems of the paper.

The $c$ grid considered in this section uses the four edge ratio functions to label edges. More specifically, given a $c$-grid we use the following equations to define labels.
\begin{equation}\label{equ:tl11}
T_{L,1,1}=1.
\end{equation}.    
\begin{align}\label{equ:edgeratiorelationships}
T_{L,r,1} &= x(c,r)T_{L,r-1,1}, &\qquad 2 \le r \le c  \nonumber  \\
T_{L,r,d} &= y(r,d) T_{L,r,d-1}, &\qquad 2 \le d  \le r.\nonumber  \\ 
T_{R,r,d} &= r_{R,L}(r,d) T_{L,r,d}, &\qquad 1 \le r \le c, 1 \le d \le r  \\ 
T_{B,r,d} &= r_{B,L}(c, r,d) T_{L,r,d}, &\qquad 1 \le r \le c, 1 \le d \le r. \nonumber 
 \end{align}
In the sequel, we will refer to \eqref{equ:edgeratiorelationships} as \textit{the (four) edge-ratio relationships.}
Figure \ref{fig:3rgrid} illustrates the application of \eqref{equ:tl11}-\eqref{equ:edgeratiorelationships}   to completely label a $3$-grid.

\begin{figure}[ht!]
\begin{center}
\caption{A $3$ grid labeled using  \eqref{equ:tl11}-\eqref{equ:edgeratiorelationships}}. 
\label{fig:3rgrid}
\begin{tabular}{|c|}

\hline
 
\begin{tikzpicture}[xscale=1.25,yscale=1.25]
 
\draw (0,0)--(1,1)--(2,2)--(3,3);
\draw (2,0)--(3,1)--(4,2);
\draw (4,0)--(5,1);

\draw (3,3)--(4,2)--(5,1)--(6,0);
\draw (2,2)--(3,1)--(4,0);
\draw (1,1)--(2,0);

\draw (0,0)--(2,0)--(4,0)--(6,0);
\draw (1,1)--(3,1)--(5,1);
\draw (2,2)--(4,2);

\node  [above left] at (2.5,2.5) {1}; 

\node  [above left] at (1.5,1.5) {$\frac{5}{6}$};
\node  [above left] at (.5,.5) {$1$};

\node  [above left] at (3.6,1.5) {$\frac{5}{2}$};
\node  [above left] at (2.6,.5) {$\frac{5}{2}$};
\node  [above left] at (4.6,.5) {$5$};

\node  [above right] at (3.5,2.5) {$1$};
\node  [above] at (3,2) {$5$};

\node  [above right] at (2.5,1.5) {$\frac{5}{2}$};
\node  [above] at (2,1) {$\frac{5}{2}$};

\node  [above right] at (4.5,1.5) {$\frac{5}{6}$};
\node  [above] at (4,1) {$\frac{5}{2}$};

\node  [above right] at (1.5,.5) {$5$};
\node  [above] at (1,0) {$1$};

\node  [above right] at (3.5,.5) {$\frac{5}{2}$};
\node  [above] at (3,0) {$\frac{5}{6}$};

\node  [above right] at (5.5,.5) {$1$};
\node  [above] at (5,0) {$1$};


\end{tikzpicture}\\
\hline
\end{tabular}
\end{center}
\end{figure}

  An important point, is that the four edge-ratio functions can be used to describe the ratio of any two edges in the 3-grid. To illustrate this, we introduce the heuristic of \textit{travel} along the edges of the $c$-grid. Heuristically, the edge ratio functions $x,y,r_{R,L}, r_{B,L}$ are used to describe vertical, horizontal, and rotational travel, respectively. The next two examples clarify how this concept of \textit{travel} will be used. 
\begin{example} Referring to Figure \ref{fig:3rgrid}
$$T_{R,3,3} =r_{R,L}(3,3) \cdot \biggl( y(3,2) y(3,3) \biggr) \cdot \biggl(x(3,2) x(3,3) \biggr) \cdot  T_{L,1,1}.$$
In this equation we start with $T_{L,1,1}$ then \textit{travel} down the left side of the grid 
to $T_{L,3,1}$  by multiplying $T_{L,1,1}$ by $x(3,2) x(3,3),$ then \textit{travel} horizontally across the bottom row from $T_{L,3,1}$ to $T_{L,3,2}$ by multiplying by $ y(3,2) y(3,3)$ and finally travel from $T_{L,3,3}$ to $T_{R,3,3}$ by multiplying by $r_{R,L}(3,3).$ \end{example}

\begin{example} The next example will be used to motivate the definition of the function $z$  in Section \ref{sec:neededfunctions}.
\begin{equation}\label{equ:zmotivation} T_{L,3,2} =
y(3,2) \cdot x(3,3) \cdot  \frac{1}{y(2,2)}
\cdot  T_{L,2,2}.\end{equation}
In this example, we start at $T_{L,2,2},$ \textit{travel} horizontally backward to $T_{L,2,1}$ by multiplying by  $\frac{1}{y(2,2)},$ then \textit{travel} vertically downward one row to $T_{L,3,1}$ using $x(3,3),$ and then \textit{travel} horizontally on row three from $T_{L,3,1}$ to $T_{L,3,2}$ using $y(3,2).$ \end{example}

In the sequel, when giving such equations, we may simply say 
\textit{they are justified by general travel considerations.}

Figure \ref{fig:3rgrid} motivates the following  formal definition of  the \textit{initial non-one $c$ grid}
and related concepts.  

\begin{definition}\label{def:initialcgrid} Given $c \ge 1,$  an \textbf{initial non-one $c$ grid} is 
a triangular $c$-grid, satisfying Definition \ref{def:ngrid}, with the edge labels satisfying \eqref{equ:tl11}-\eqref{equ:edgeratiorelationships}.

A \textbf{non-one $c$ grid} (without the adjective initial)
is  a triangular $c$-grid, satisfying Definition \ref{def:ngrid}, with the edge labels satisfying
the four edge-ratio relationships.  

An \textbf{arbitrary $c$-grid } is  a triangular $c$-grid, satisfying Definition \ref{def:ngrid}.
\end{definition}

In the rest of the paper, an initial non-one $c$ grid is notationaly indicated by $T^{c};$  
$T^{c,i}, 1 \le i \le c$  indicates the non-one $i$ grid obtained by applying $c-i$ consecutive row-reductions to $T^{c}$ (note, $T^c = T^{c,c}.)$

 We can further use this example to illustrate all three of the Main Theorems. Figure \ref{fig:2reductions}  shows the results of applying the row-reduction algorithm, Definition \ref{def:rr}, or the four local  functions,  to the initial non-one 3-grid of Figure \ref{fig:3rgrid}. The computations are similar to those presented in Figure \ref{fig:12panels}.

\begin{figure}[ht!]
\begin{center}
 \caption{
 Two applications of row reduction  
to an initial non-one 3-grid (Figure \ref{fig:3rgrid}) resulting in a 
non-one 2-grid and 1 grid.}  \label{fig:2reductions}
\begin{tabular}{||c|c||} 
\hline
\begin{tikzpicture}
[xscale=1.75,yscale=1.75]

\draw (0,0)--(1,1)--(2,2) ;
\draw (2,0)--(3,1) ;

\draw (2,2)--(3,1)--(4,0);
\draw (1,1)--(2,0);

\draw (0,0)--(2,0)--(4,0);
\draw (1,1)--(3,1); 

\node  [above left] at (1.5,1.5) {$\frac{15}{14}$};
\node  [above left] at (.5,.5) {$\frac{15}{14}$};
\node  [above left] at (2.5,.5) {$\frac{45}{14}$};

\node  [above right] at (2.5,1.5) {$\frac{15}{14}$};
\node  [above right] at (3.5,.5) {$\frac{15}{14}$};
\node  [above right] at (1.5,.5) {$\frac{45}{14}$};

\node  [above] at (2,1) {$\frac{45}{14}$};
\node  [above] at (1,0) {$\frac{15}{14}$};
\node  [above] at (3,0) {$\frac{15}{14}$};

\node [below] at (2,-.5) {Panel A};


\end{tikzpicture}

&

\begin{tikzpicture}
[xscale=1.75,yscale=1.75]
 
\draw (0,0)--(1,1)--(2,0);
\draw (0,0)--(2,0);

\node  [above left] at (.5,.5) {$\frac{9}{7}$}; 
\node  [above right] at (1.4,.5) {$\frac{9}{7}$};
\node  [above left] at (1,0) {$\frac{9}{7}$};
\node [below] at (1,-.5) {Panel B};

\end{tikzpicture}

\\
\hline
\end{tabular}

\end{center}
\end{figure}

We observe the following, each of which will be generalized  and proven later in the paper. The first main result generalizes the following example. For any $X \in \{L,R,B\},$ we have
 \begin{align} \label{equ:tailratios}
T^{3,2}_{X,r,d} &= \frac{15}{14} T^{2}_{X,r,d}, &\qquad 1 \le r \le 2; 1 \le d \le r \nonumber \\
T^{3,1}_{X,r,d} &= \frac{9}{7} T^{1}_{X,r,d}, &\qquad 1 \le r \le 1; 1 \le d \le r. \end{align}
Equation \eqref{equ:tailratios} immediately implies the following result which the second main result generalizes:     The reduced grids $T^{3,i},i=1,2$ satisfy the four edge-ratio relationships.

By Definition \ref{def:tail} as illustrated
by Example \ref{exa:tail} we have 
\begin{equation}\label{equ:tails}
	 t(3,3)=\Delta(1,1,5) = \frac{1}{7}; \qquad
    	t(3,2)=\Delta(\frac{15}{14},\frac{15}{14},\frac{45}{14})=\frac{3}{14}; \qquad
	 t(3,1)=\Delta(\frac{9}{7},\frac{9}{7},\frac{9}{7})=\frac{3}{7}. 
\end{equation}
This implies
$$
   t(3,2)=\frac{1}{2} \cdot t(3,1) \text{ and }
    t(3,3)=\frac{1}{3} \cdot t(3,1), 
$$
relationships which mirror those found in Conjecture \ref{con:con2}. The third main result makes this resemblance precise and generalizes it.

\section{Needed Functions}\label{sec:neededfunctions}

Besides the four local functions, to formulate and prove the three main theorems we need three additional functions, $f, g, z,$ presented in this section.

For positive integer $c,$ define
\begin{equation}\label{equ:functionf}
		f(c,i) = 1+\frac{c-i}{i}\frac{1}{2c+1}, 1 \le i \le c-1, \qquad g(c) =  \frac{c}{c-1} \frac{2c-1}{2c+1}.
\end{equation}
\begin{example} $f(3,2)=\frac{15}{14}, f(3,1)=\frac{9}{7}$ corresponding to the ratios found in \eqref{equ:tailratios}. We will formally prove in one of the main theorems that $f(c,i)=\frac{T^{c,i}_{L,1,1}}{T^{i}_{L,1,1}}.$ To motivate  Proposition \ref{lem:gf}, one may verify that $\frac{15}{14}=g(3)=f(3,2),$  $\frac{9}{7}=f(3,1) = g(3) \cdot g(2)=
\frac{15}{14} \cdot \frac{6}{5}.$ \end{example}

\begin{proposition}\label{lem:gf} For integer $c \ge 1$ and        $1 \le j \le c-1,$ 
\begin{equation}\label{equ:temp1} f(c,j) = \displaystyle \prod_{i=j+1}^c g(i).\end{equation} \end{proposition}

\begin{proof} The proof is by induction on $j$ with arbitrary $c$ fixed. For the base case, $j=c-1,$ we verify that $f(c,c-1)=g(c).$
For an induction assumption we assume that for some $s, 1 \le s \le c-2,$ that 
\begin{equation}\label{temp2} 
	f(c,s) = \displaystyle \prod_{i=s+1}^c g(i).
\end{equation}
For an induction step, to complete the proof, we need to prove
\begin{equation}\label{temp3} 
	f(c,s+1) = \displaystyle \prod_{i=s+2}^c g(i).
\end{equation}
Dividing both sides of \eqref{temp2} and \eqref{temp3} we see that it suffices to verify
\begin{equation}
	\frac{f(c,s)}{f(c,s+1)}  = g(s+1).
\end{equation}
This verification is an assertion that two rational functions are equal which can be verified manually or by algebraic software. (Recall (Section \ref{sec:algorithm},) that to facilitate the flow of the narrative, proofs in this paper are accomplished by reducing assertions to the verification that  two rational functions are equal, or equivalently that their difference is 0 or their ratio is 1.)
Since $c$  was arbitrary, this completes the proof. 
\end{proof} 

 Define a function $z,$ by
\begin{equation}\label{equ:zoriginal}
		z(r,d) =\frac{\prod_{i=2}^d y(r+1,i)}{\prod_{i=2}^d y(r,i)},
  \text{ for } 1 \le r \le c-1, 2 \le d \le r.
\end{equation}
As illustrated in \eqref{equ:zmotivation} and the surrounding narrative, we have
\begin{equation}\label{equ:zlemma}
   T_{L,r+1,d} = x(c,r+1) z(r,d)
    T_{L,r,d}.
\end{equation}

However $z$ as defined in \eqref{equ:zoriginal} is not closed which would not allow
using algebraic verification as a proof method since this method requires closed forms. 
Fortunately, there is an alternate closed form for $z(r,d).$

\begin{proposition}\label{pro:z} With $z$ defined by \eqref{equ:zoriginal}, for $1 \le d \le r,$	
\begin{equation}\label{equ:zclosed}
	z(r,d)=1 -\frac{d-1}{r} \times
	        \frac{1}{2(r-d)+3} \doteq z_1(r,d),
\end{equation}
where we have defined $z_1(r,d)$ as indicated.
\end{proposition}
\begin{proof}
The proof is by induction on $d$ with an arbitrary $r$ fixed. For a base case,   $d=2,$ by  \eqref{equ:zclosed}, \eqref{equ:zoriginal},
and \eqref{equ:edgefactorshorizontal}, we may verify $z(r,2) =\frac{y(r+1,2)}{y(r,2)}=z_1(r,2).$

For an induction assumption,
we assume the following  holds for some $s, 2 \le s \le r-1.$
\begin{equation}\label{equ:zTemp1}
	\frac{\prod_{i=2}^s y(r+1,i)}{\prod_{i=2}^s y(r,i)}.
 =z_1(r,s)  
\end{equation}

For the induction step, which will complete the proof, we must, given \eqref{equ:zTemp1}, prove
\begin{equation}\label{equ:zTemp2}
	\frac{\prod_{i=2}^{s+1} y(r+1,i)}{\prod_{i=2}^{s+1} y(r,i)}
 =z_1(r,s+1).
\end{equation}

Dividing the right and left-hand sides of \eqref{equ:zTemp1} and \eqref{equ:zTemp2} (and performing routine cancellations) we see that it suffices to verify
$\frac{y(r+1,s+1)}{y(r,s+1)}=z(r,s+1)/z(r,s).$
This completes the proof.
\end{proof}

In the sequel we will use $z$ to refer either to $z$ proper or to $z_1.$

\section{The Three Main Theorems}\label{sec:maintheorems}

  Section \ref{sec:motivatingexample} motivated with illustrative examples the main theorems of this paper. In stating the main theorems we use the language presented in Definition \ref{def:initialcgrid} and the notation indicated immediately after it.

To state the theorem about the symmetry of reduced initial $c$ grids we have to formally define
what we mean by vertical, rotational, and slide symmetry. We provide a graph theoretic definition, independent of an embedding of the graph  in the Cartesian plane.

\begin{definition}\label{def:isotropy} An arbitrary $c$ grid  is said to possess vertical symmetry if
$$
	T_{L,r,d} = T_{R,r+1-d,2} \qquad 
	T_{B,r,d} = T_{B,r+1-d,2}, \qquad
	1 \le d \le r \le c.
$$
An arbitrary $c$-grid is said to possess rotational symmetry (by $\frac{\pi}{3}$) if  for
$1 \le d \le r \le c,$
$$
    T_{L,r,d} = T_{R,n+d-r,n+1-r},\qquad
	T_{R,r,d} = T_{B,n+d-r,n+1-r} ,\qquad
	T_{B,r,d} = T_{L,n+d-r,n+1-r}.  
$$
An arbitrary $c$ grid is said to posses  \emph{slide} symmetry if
for   $1 \le d \le r \le c,$
$$
    T_{L,r,d} = T_{L,n+d-r,d},  \qquad 
    T_{R,r,d} = T_{B,n+d-r,d}  
$$
\end{definition}
These symmetries are amply illustrated in Figures \ref{fig:3rgrid} and \ref{fig:2reductions}.

\begin{remark} It is straightforward to verify that an assumption of vertical and slide symmetry is equivalent to an assumption of vertical and rotational symmetry.
Therefore, in the sequel the statement that a $c$-grid possesses symmetry will mean that it possess vertical, rotational, and slide symmetry.
\end{remark}
 
\begin{theorem}[Symmetry] \label{the:isotropy} For integer $c>0, 1 \le i \le c$  $T^{c,i}$
possesses vertical and slide symmetry. \end{theorem}.

\begin{theorem}[Edge ratios]\label{the:edgefactors}  
For integer $c>0,$ and an initial non-one $c$ grid $T^{c}.$  \\
i) For $1 \le i \le c,$ 
\begin{equation}\label{equ:trci}
    t(c,i) = \frac{1}{i} \cdot t(c,1)
\end{equation}
with
\begin{equation}\label{equ:trc1}
    t(c,1)= \frac{c}{2c+1}.
\end{equation}
ii) For $1 \le i \le c$ the four edge-factor relationships hold in $T^{c,i}.$ \\
iii) For $1 \le i \le c-1,$ 
\begin{equation}\label{equ:fcttrc}
    T^{c,i}_{X,r,d}  = f(c,i) \cdot T^{i}_{X,r,d}, \text{ for } 
	1 \le r \le i, 1 \le d \le r, X \in \{L,R,B\}.
\end{equation}
\end{theorem}

To formulate the third main theorem we need to introduce the concept of the upper-left half of a $c$ grid.

 \begin{definition}\label{def:upperlefthalf}  The \emph{upper left half} of an arbitrary $c$-grid, consists of the triangles
\begin{equation}\label{equ:upperlefthalf}
	T_{r,d}, \qquad  d=1,\dotsc, \lfloor \frac{c+2}{3} \rfloor,  r =2d-1, \dotsc, \lfloor \frac{c+d}{2} \rfloor.
\end{equation} 
\end{definition}

Figure \ref{fig:upperlefthalf} illustrates the definition.

\begin{figure}[ht!]
\begin{center}
\caption{ 
The \textit{upper left half} of this 4-grid consists of the triangles $T_{1,1}, T_{2,1},T_{3,2}.$   
}\label{fig:upperlefthalf}
\begin{tabular} {|c|}
\hline

\begin{tikzpicture} [xscale=1,yscale =1]

\node [below] at (3,3.2) {    };
 
\draw (-1,-1)--(0,0)--(1,1)--(2,2)--(3,3);
\draw (1,-1)--(2,0)--(3,1)--(4,2);
\draw (3,-1)--(4,0)--(5,1);
\draw (5,-1)--(6,0);

\draw (3,3)--(4,2)--(5,1)--(6,0)--(7,-1);
\draw (2,2)--(3,1)--(4,0)--(5,-1);
\draw (1,1)--(2,0)--(3,-1);

\draw (-1,-1)--(1,-1)--(3,-1)--(5,-1)--(7,-1);
\draw (0,0)--(2,0)--(4,0)--(6,0);
\draw (1,1)--(3,1)--(5,1);
\draw (2,2) -- (4,2);

\begin{scriptsize}
\node [above] at (0,-.8)
{$T_{4,1}$};
\node [above] at (2,-.8)
{$T_{4,2}$};
\node [above] at (4,-.8)
{$T_{4,3}$};
\node [above] at (6,-.8)
{$T_{4,4}$};

\node  [above] at (1,.2) {$T_{3,1}$};
\node  [above] at (3,.2) {$T_{3,2}$};
\node  [above] at (5,.2) {$T_{3,3}$};
\node  [above] at (2,1.2) {$T_{2,1}$};
\node  [above] at (4,1.2) {$T_{2,2}$};
\node  [above] at (3,2.3) {$T_{1,1}$};

\node [below] at (6,0) {$ $};
\node [below] at (4,0) {$ $};
\node [below] at (2,0) {$ $};
\node [below] at (0,0) {$ $};

\end{scriptsize}

\end{tikzpicture}
\\ \hline
\end{tabular}
\end{center}
\end{figure}

\begin{remark} In the sequel it will be convenient to be able to refer to the i) upper right half, ii) lower left half and iii) lower right half of an arbitrary $c$-grid. These respectively correspond to the triangles $T_{r,d}$ with $(r,d)$ in 
i) $d=\lfloor \frac{c+2}{3} \rfloor,\dotsc,r,$  $r =2d-1, \dotsc, \lfloor \frac{c+d}{2} \rfloor,$
ii) $d=1,\dotsc, \lfloor \frac{c+2}{3} \rfloor,$  $r =\lfloor \frac{c+d}{2} \rfloor, \dotsc, c,$
iii) $d=\lfloor \frac{c+2}{3} \rfloor,\dotsc,r,$  $r =\lfloor \frac{c+d}{2} \rfloor, \dotsc, c.$
Notice that these regions are not mutually disjoint. We may further use the terms i) top corner, ii) 4 O'Clock corner, and iii) 8 O'Clock corner of the $c$-grid  to respectivelly refer to i) the union of the upper left and right half, ii) the lower right half, and iii) the lower left half of an arbitrary  $c$ grid. 
\end{remark}

\begin{proposition}\label{lem:upperlefthalf2} The  edge labels in the upper left half of an arbitrary $c$-grid possessing  symmetry 
determine all remaining edge-labels in the $c$-grid. \end{proposition}
\begin{proof} Vertical symmetry maps the \textit{upper left half} to \textit{the upper right half} showing that the edges in the \textit{top corner} of the $c$-grid  are determined. Two applications of rotational symmetry then map this top corner to the  4 O'Clock and 8 O'Clock corners showing that the edge labels of the entire triangular grid are determined from the edge labels in the upper left half. 
\end{proof}

We provide one more definition.

\begin{definition} Let $T$ be a $c$-grid. 
For $1 \le t \le \lfloor \frac{c+2}{3} \rfloor,$ the $t$-rim of a $c$-grid
is a union of the following 3 sets of triangles:
$T_{r,1},2t-1 \le r \le c+1-t,$
$T_{c+1-t,d}, t \le d \le c-2(t-1)  $
$T_{r,r-(t-1)},2t-1 \le r \le c+1-2t $.
\end{definition}

\begin{example}If $c=4,$ as in Figure \ref{fig:upperlefthalf},
then the 2-rim consists of the singleton triangle $T_{3,2}$
while the 1-rim consists of the 9 triangles
$\{T_{r,1}\}_{1 \le r \le 4} \cup \{T_{4,d}\}_{1 \le d \le 4} \cup \{T_{r,r}\}_{1 \le r \le 4}.$  Heuristically, the $c$ grid may be perceived as a collection of concentric rims.  
\end{example} 

We may now state the third main theorem.

\begin{theorem}[Edge-Ratio Symmetry duality] \label{the:symmetryedgefactors} For  an arbitrary $c$ grid possessing symmetry, if the four edge-ratio relationships are satisfied for edges of triangles  in the upper left half, then the four edge-factor relationships are satisfied in the entire $c$-grid. 
\end{theorem}

The name given to the theorem reflects the heuristic of its insights. The theorem coupled with 
 Proposition \ref{lem:upperlefthalf2} intuitively says that the local edge-ratio relationships mirror global symmetry in the sense that (with the conditions as indicated in the theorem) the edge ratio relationships imply symmetry and symmetry in turn implies the edge-factor relationships.

Section \ref{sec:motivatingexample} presented the non-one initial $3$-grid. Similar methods can be used to calculate the initial non-one $c$-grids, $c=1,2,4.$ The theorems of this paper can be manually checked for these small values.

\begin{proposition}\label{lem:cgreaterthan4} To prove Theorem \ref{the:edgefactors}(ii) it suffices to prove the required
equations for any $c$ grid with $c \ge 5.$ \end{proposition}

An advantage of restricting ourselves to $c \ge 5$ is the following proposition.

\begin{proposition}\label{lem:4localfunctions}
For $c \ge 5$ the restrictions on using the four local functions are satisfied in the upper left half of the initial $c$-grid.
\end{proposition}
\begin{proof}
By \eqref{equ:perimeteredge}-\eqref{equ:leftedge}, to use the four local functions,   requires, $$
r \le c-2, \qquad d \le r-1.
$$
By Definition \ref{def:upperlefthalf}, $(r,d)$ lies   in the upper left half if
\begin{equation*} 
	T_{r,d}, \qquad  d=1,\dotsc, \lfloor \frac{c+2}{3} \rfloor,  r =2d-1, \dotsc, \lfloor \frac{c+d}{2} \rfloor.
\end{equation*} 
It is straightforward to check that the solution to these simultaneous inequalities is
$c \ge 5.$
\end{proof}

\section{Proof of the Symmetry Theorem}\label{sec:isotropy}

Theorem \ref{the:isotropy} states that both $T^{c}$ and its reductions $T^{c,i}, 1 \le i \le c$ possess  symmetry. However, in terms of proofs, we first prove the theorem for the case $i=c.$ We then, 
(Proposition \ref{lem:inheritance}),  prove that the reduced grids inherit, one reduction at a time, the symmetries of the parent grid, completing the proof of the theorem. 

By Definition \ref{def:isotropy}, to prove  Theorem \ref{the:isotropy} for the case $i=c$ we must prove
vertical  and slide symmetry for  $T^{c}.$

\textbf{Proof of vertical symmetry.} We prove 
\begin{equation}\label{VTemp1}
	T_{L,r, \frac{r}{2}-s} = T_{R,r,\frac{r}{2}+s+1},
 \text{ for } 0 \le s \le \frac{r}{2}-1,
 \end{equation}
the corresponding proof for the base edges being similar and hence omitted. The proof, depends on the parity of $r.$ We assume $r$ even, the proof for
the $r$ odd case being similar and hence omitted.

For the base case we prove \eqref{VTemp1} for the case $s=0.$ Travel from
$T_{L,r, \frac{r}{2}-s}$ to $T_{L,r,\frac{r}{2}+s+1}$ is accomplished by $y$ while
travel from $T_{L,r,\frac{r}{2}+s+1},$ to $T_{R,r,\frac{r}{2}+s+1},$ is accomplished by $r_{R,L}.$ Hence the proof for the case $s=0$ is completed by verifying 
\begin{equation}\label{ITemp1}
		y(r,\frac{r}{2}+1) \cdot r_{R,L}(r,\frac{r}{2}+1) =1.
\end{equation}

For an induction assumption we assume \eqref{VTemp1} true for some $s \ge 0.$ Then, to complete the proof, we must prove an induction step, that  given \eqref{VTemp1} we have
\begin{equation}\label{VTemp2}
		T_{L,r, \frac{r}{2}-s-1} = T_{R,r,\frac{r}{2}+s+2}.
\end{equation}
The travel from 
i) $T_{L,r, \frac{r}{2}-s-1}$ to $T_{L,r, \frac{r}{2}-s}$,  then from 
ii) $T_{L,r, \frac{r}{2}-s}$, to  $T_{R,r,\frac{r}{2}+s+1}$, then iii) from $T_{R,r,\frac{r}{2}+s+1}$ to $T_{L,r,\frac{r}{2}+s+1},$ 
then iv) from $T_{L,r,\frac{r}{2}+s+1},$ to
$T_{L,r,\frac{r}{2}+s+2},$ and finally 
(v) from $T_{L,r,\frac{r}{2}+s+2},$ to $T_{R,r,\frac{r}{2}+s+2},$
 is respectively 
accomplished by 
i) $\frac{1}{y(r,\frac{r}{2}-S)}$
ii) $y(r,\frac{r}{2}+1) \cdot r_{R,L}(r,\frac{r}{2}+1)$ 
(which by \eqref{ITemp1} equals 1),
iii) $\frac{1}{r_{R,L}(r,\frac{r}{2}+s+1)},$
iv) $y(r, \frac{r}{2}+s+2),$ and
v) $r_{R,L}(r,\frac{r}{2}+s+2).$
Consequently the proof of \eqref{VTemp2} given \eqref{ITemp1} is completed by verification that 
$$
	\frac{1}{y(r, r/2 - s)} =     
  	\frac{1}{r_{R,L}(r, r/2 + s + 1)}
	 y(r, r/2 + s + 2) 
	r_{R,L}(r, r/2 + s + 2).
$$

\textbf{Proof of Slide Symmetry}. 
By Definition \ref{def:isotropy} we must prove slide symmetry for all three sides. We suffice with dealing just with the left-hand side, that is, with proving,
\begin{equation}\label{ITemp4}
	 T_{L,\frac{c+d-1}{2}-i,d} = T_{L,\frac{c+d-1}{2}+1+i,d}, \qquad
	\text{ for } 1 \le d \le c, 0 \le i \le \frac{c-d-1}{2},
\end{equation}
the proof for the other sides being similar and hence omitted. In the proof, we assume $c+d$ odd, the proof for the even case being similar and hence omitted.  
 	 
For the base case we prove \eqref{ITemp4} for $i=0.$ By \eqref{equ:zlemma}  
 it suffices to verify
$$
	x(c,\frac{c+d-1}{2}+1) \cdot z(\frac{c+d-1}{2},d)=1.
$$

\textbf{Induction Assumption.} For an induction assumption, we assume that for some $s, 0 \le s \le \frac{c-d-1}{2}-1$ 
that 
\begin{equation}\label{equ:ITemp5} 
	  T_{L,\frac{c+d-1}{2}-s,d} = T_{L,\frac{c+d-1}{2}+1+s,d}.
\end{equation} 

 To complete the proof, we must, using an  induction step,  prove
\begin{equation}\label{ITemp6} 
	  T_{L,\frac{c+d-1}{2}-s-1,d} = T_{L,\frac{c+d-1}{2}+2+s,d}.
\end{equation} 
Travel  i) from $T_{L,\frac{c+d-1}{2}-s-1,d}$ to  $T_{L,\frac{c+d-1}{2}-s,d},$ 
then ii) to $T_{L,\frac{c+d-1}{2}+1+s,d},$
then iii) to $T_{L,\frac{c+d-1}{2}+2+s,d},$ 
is, by \eqref{equ:zlemma}, respectively accomplished by 
i) $\biggl(x(c,\frac{c+d-1}{2}-s) \cdot 
z(\frac{c+d-1}{2}-s-1,d) \biggr),
$
ii) 1 (by the induction assumption, \eqref{equ:ITemp5}), and
iii) $\biggl( x(c,\frac{c+d-1}{2}+s+2) \cdot z(\frac{c+d-1}{2}+(s+1),d)\biggr).$
Hence, the proof is completed by verifying 
$$
\biggl(x(c,\frac{c+d-1}{2}-s) \cdot 
z(\frac{c+d-1}{2}-s-1,d) \biggr) \cdot
\biggl( x(c,\frac{c+d-1}{2}+s+2) \cdot z(\frac{c+d-1}{2}+(s+1),d)\biggr)  
=1.
$$

This competes the  proof of slide symmetry and  the symmetry theorem for the case $i=c.$ To prove the symmetry theorem for all $i$ we need the following  Proposition.

\begin{proposition}\label{lem:inheritance} If an arbitrary $c$ grid possesses  symmetry then the
$c-1$ grid arising from one reduction of it also possesses symmetry.
\end{proposition}
\begin{proof} In the obvious manner we may speak about a triangular grid of $Y$s possessing vertical, rotational,  and slide symmetry. Since the $Y$s arising from a $\Delta$ - Y transformation are determined by
the sides of the underlying triangle, the $Y$s will have the same symmetries as the original
configuration of triangles. Walking through the definition of reduction, Definition \ref{def:rr}, we see that discarding tails, performing series transformation on adjacent edges, and performing Y-$\Delta$ transformations preserve vertical and slide symmetry. Hence the once-reduced grid will possess the symmetries of the parent grid.
\end{proof}

\begin{corollary} 
The symmetry theorem is true for all $i, 1 \le i \le c.$
\end{corollary}

\section{Proof of the Edge-Factor Symmetry Duality Theorem }
\label{sec:inheritance}

   The proof 
of Theorem \ref{the:symmetryedgefactors}
consists of a series of  propositions and corollaries, each  proposition-corollary pair focusing on one edge-factor relationship.
Throughout this section statements of satisfaction of the edge-ratio relationships or of edges belonging to the upper left half have the formal meaning given in
\eqref{equ:edgeratiorelationships} and
Definition 
\ref{def:upperlefthalf} respectively.

\begin{proposition}\label{lem:horizontalduality} For an arbitrary $c$ grid possessing  symmetry and for an arbitrary fixed row, $r, 1 \le r \le c:$ \\
a)   If the \textbf{right-left} edge ratio relationship $\frac{T_{r,d,R}}{T_{r,d,L}}= r_{R,L}(r,d)$ holds for $1 \le d \le \lfloor \frac{r+1}{2}  \rfloor,$ then it holds for $1 \le d \le r.$\\
b)  If the \textbf{base-left} edge ratio relationship $\frac{T_{r,d,B}}{T_{r,d,L}}= r_{B,L}(r,d)$ holds for $1 \le d \le \lfloor \frac{r+1}{2} \rfloor$, then it  holds for  $1 \le d \le r.$\\
c)  If the \textbf{horizontal} edge ratio relationship $\frac{T_{r,d,L}}{T_{r,d,L}}= y(r,d)$ holds for $2 \le d \le \lfloor \frac{r+1}{2} \rfloor,$ then it  holds for $1 \le d \le r.$\\
\end{proposition}
\begin{proof}
a) Vertical symmetry implies
that the numerators and denominators of the following two fractions are equal (and hence their ratios are equal).
$$
	\frac{T_{r,d,R}}{T_{r,d,L}} = \frac{T_{r,r+1-d,L}}{T_{r,r+1-d,R}}.
$$
By the hypothesis of the  Proposition that the edge-ratio relationships hold in the upper left half, for
$1 \le d \le \lfloor \frac{r+1}{2} \rfloor,$
$$
	r_{R,L}(r,d) =\frac{T_{r,d,R}}{T_{r,d,L}}.
$$
We can algebraically verify that for all $d, 1 \le d \le r,$
$$
	r_{R,L}(r,d) = \frac{1}{r_{R,L}(r,r+1-d)}.
$$
Combining these equations we infer
$$
 \frac{T_{r,r+1-d,R}}{T_{r,r+1-d,L}} = r_{R,L}(r,r+1-d)
$$
as required.\\

b) Vertical symmetry implies
$$
	\frac{T_{r,d,B}}{T_{r,d,L}}=\frac{T_{r,r+1-d,B}}{T_{r,r+1-d,R}}=
	\frac{T_{r,r+1-d,B}}{T_{r,r+1-d,L}}\frac{T_{r,r+1-d,L}}{T_{r,r+1-d,R}}.
$$
By the hypothesis of the  Proposition,
for
$1 \le d \le \lfloor \frac{r+1}{2} \rfloor,$
$$
		\frac{T_{r,d,B}}{T_{r,d,L}} = r_{B,L}(c,r,d).	
$$
By part (a)
$$
	\frac{T_{r,r+1-d,L}}{T_{r,r+1-d,R}}= \frac{1}{r_{R,L}(r,r+1-d)}.
$$

We can algebraically verify
$$
 	r_{B,L}(c,r,d) =r_{B,L}(c,r,r+1-d) \frac{1}{r_{R,L}(r,r+1-d)}.
$$
Combining these equations we infer, as required, that
$$
	\frac{T_{r,r+1-d,B}}{T_{r,r+1-d,L}}=r_{B,L}(c,r,r+1-d).	
$$\\

(c) Vertical symmetry implies 
$$
		\frac{T_{r,d,L}}{T_{r,d-1,L}} = \frac{T_{r,r+1-d,R}}{T_{r,r+2-d,R}} 
	= \frac{T_{r,r+1-d,L}}{T_{r,r+2-d,L}}\frac{T_{r,r+1-d,R}}{T_{r,r+1-d,L}}
		\frac{T_{r,r+2-d,L}}{T_{r,r+2-d,R}}.
$$

By the hypothesis of the  Proposition and for $d$ in the range indicated
$$
	y(r,d)=\frac{T_{r,d,L}}{T_{r,d-1,L}}.
$$
By Part (a)
$$
	\frac{T_{r,r+1-d,R}}{T_{r,r+1-d,L}}
		\frac{T_{r,r+2-d,L}}{T_{r,r+2-d,R}} =\frac{r_{R,L}(r,r+1-d)} {r_{R,L}(r,r+2-d)}.
$$
We can algebraically verify
$$
	y(r,d)= \frac{1}{y(r,r+2-d)} \frac{r_{R,L}(r,r+1-d)}{r_{R,L}(r,r+2-d)}.
$$
Combining these last 3 equations we infer,as required, that
$$
	\frac{T_{r,r+2-d,L}}{T_{r,r+1-d,L}} = y(r,r+2-d).
$$  \end{proof}

\begin{corollary}
For an arbitrary $c$ grid possessing  symmetry, if the right-left, base-left, and horizontal edge ratios hold in the upper left half then they hold in the entire upper half.
\end{corollary}

\begin{proposition}\label{lem:upperlefthalflefthalf}
For an arbitrary $c$ grid possessing symmetry:\\
a)   If the vertical edge ratio relationship holds in the upper half, then
it holds in the entire $c$ grid.\\
 b)  If the right-left and base-left edge ratio  relationships hold in the upper half, then
they hold in the entire $c$ grid. \\
\end{proposition}
\begin{proof}  
a) Slide symmetry applied to diagonal $d=1,$ implies
$\frac{T_{r,1,L}}{T_{r-1,1,L}} = \frac{T_{c+1-r,1,L}}{T_{c+2-r,1,L}}.$
But $x(c,r) = \frac{T_{r,1,L}}{T_{r-1,1,L}}.$ We may algebraically verify
$x(c,r)=\frac{1}{x(c,c+2-r))}, 2 \le r \le c.$ Combining these equations
we infer $\frac{T_{c+2-r,1,L}}{T_{c+1-r ,1,L}}=x(c,c+2-r)$ as required.\\
b)For any diagonal $d,$ slide symmetry requires that for $d \le r \le c,$
\begin{equation}\label{82bc1}\frac{T_{r,d,R}}{T_{r,d,L}} = \frac{T_{c+d-r,d,B}}{T_{c+d-r,d,L}}.
\end{equation}
We may algebraically verify \begin{equation}\label{82bc2}r_{R,L}(r,d) = r_{B,L}(c,c+d-r,d).\end{equation} We can now  combine these two equations in two ways. \\
(i) If $r$ is in the upper half, then by hypothesis,
$r_{R,L}(r,d) =\frac{T_{r,d,R}}{T_{r,d,L}}.$  Combining this equation with \eqref{82bc1}-\eqref{82bc2} we infer $r_{B,L}(c,c+d-r,d)=
\frac{T_{c+d-r,d,B}}{T_{c+d-r,d,L}},$  with $c+d-r$ in the lower half as required.\\
(ii) On the other hand, if $r$ is in the lower half, then $c+d-r$ is in the upper half and by hypothesis $r_{B,L}(c,c+d-r,d) =\frac{T_{c+d-r,d,B}}{T_{c+d-r,d,L}}.$ Combining the two equations we infer
$r_{R,L}(r,d) = 
\frac{T_{r,d,R}}{T_{r,d,L}}$ with $r$ in the lower half as required.
\end{proof} 

\begin{corollary}
For an arbitrary $c$-grid possessing symmetry, if the four edge-ratio relationships hold in the upper left half then i) the vertical, right-left, and base-edge ratio relationships hold in the entire $c$-grid, and ii) the horizontal edge ratio relationships holds in the upper half.
\end{corollary}

Thus, to complete the proof of Theorem \ref{the:symmetryedgefactors} we must show that the horizontal edge-ratio relationship holds in the lower  half. The next  Proposition proves this inductively, the various  Proposition components corresponding to the parts of the inductive proof and themselves requiring their own inductive proofs.

\begin{proposition}\label{lem:horizontaledgefactor}
For an arbitrary $c$ grid possessing symmetry, if the edge ratio relationships hold in the upper left half, then: \\
a) The horizontal edge ratio relationship, $y(c,r)$ holds for the bottom row, $r=c.$ \\
b) $\mathbf{Induction\:Base\: Case:}$ For $d=2,$ i) the horizontal edge ratio relationships holds (that is, $y(r,2) = 
\frac{T_{r,2,L}}{T_{r,1,L}}, 2 \le r \le c$),  and ii) 
$x(c,r) \cdot z(r-1,2) = \frac{T_{c,r,2}}{T_{c,r-1,2}}$ for $3 \le r \le c.$\\
c) $\mathbf{The\: Induction\: Step.}$ For each diagonal $d>2,$ i) the horizontal edge ratio relationship holds for diagonal $d$ and ii) $x(c,r) \cdot z(r-1,d) = \frac{T_{c,r,d}}{T_{c,r-1,d}},$ for $d+1 \le r \le c.$ \\
d) The horizontal edge ratio holds for the entire $c$-grid.
\\
\end{proposition}
\begin{proof}   Clearly, d) follows from a), b), and c). Thus we need only prove a)-c).  Proposition 
\ref{lem:horizontalduality}(c) (which was stated for arbitrary $r$) asserts that to prove the horizontal edge ratio is satisfied in the entire $c$-grid it suffices to show it is satisfied in the left half. Accordingly throughout the proof we assume 
$1 \le r \le \lfloor \frac{c+2}{2} \rfloor.$  Note that the set of $r$ satisfying $1 \le r \le \lfloor \frac{c+2}{2} \rfloor$ is non-empty as, by
 Proposition \ref{lem:cgreaterthan4}, it contains 3.

\textbf{Proof of Part (a).}
By rotational symmetry,  
for $2 \le r \le c,$
$\frac{T_{r,1,L}}{ T_{r-1,1,L}} = 
\frac{T_{c,r,B}}{ T_{c,r-1,B}} =
\frac{T_{c,r,B}}{ T_{c,r,L}}   
\frac{T_{c,r,L}} {T_{c,r-1,L}}
\frac{T_{c,r-1,L}} {T_{c,r-1,B}}.$
Previous results in this section have established that
$x(c,r)= \frac{T_{r,1,L}}{ T_{r-1,1,L}}$
and $\frac{T_{c,r,B}}{ T_{c,r,L}}   
\frac{T_{c,r-1,L}} {T_{c,r-1,B}}=\frac{r_{B,L}(c,c,r)}{ r_{B,L}(c,c,r-1)}.$
Moreover, we may verify $x(c,r) = r_{B,L}(c,c,r) \frac{y(c,r)}{r_{B,L}(c,c,r-1)}.$
Combining these equations we infer $y(c,r)=\frac{T_{r,1,L}}{ T_{r-1,1,L}}$ as was to be proved.\\

\textbf{Proof of Part (b).} For the proof we first need to verify the following.
 \begin{equation}\label{104E}
x(c,r) \cdot z(r-1,2) = 
\frac{1}{x(c,c+3-r)}
\frac{1}{z(c+2-r,2)}=
\frac{1}{x(c,c+3-r)} 
\frac{y(c+2-r,2)}{y(c+3-r,2)}.
\end{equation}

By slide symmetry
\begin{equation}\label{104A}
\frac{T_{r,2,L}}{T_{r-1,2,L}}
=
\frac{T_{c+2-r,2,L}}{T_{c+3-r,2,L}}.
\end{equation}

By results previously established in this section, for the left-hand side of \eqref{104A}
 \begin{equation}\label{104B}
\frac{T_{r,2,L}}{T_{r-1,2,L}}
=
\frac{T_{r,1,L}}{T_{r-1,1,L}}
\frac{T_{r,2,L}}{T_{r,1,L}}
\frac{T_{r-1,1,L}}{T_{r-1,2,L}}
=
x(c,r) \frac{y(r,2)}{y(r-1,2)}
= x(c,r) \cdot z(r-1,2),
\end{equation}
the last equality following from   \eqref{equ:zoriginal}-\eqref{equ:zlemma}.

By results previously established in this section,
for the right-hand side of \eqref{104A},
\begin{equation}\label{104C}
\frac{T_{c+2-r,2,L}}{T_{c+3-r,2,L}} =
\frac{T_{c+2-r,1,L}}{T_{c+3-r,1,L}}
\frac{T_{c+3-r,1,L}}{T_{c+3-r,2,L}}
\frac{T_{c+2-r,2,L}}{T_{c+2-r, 1,L}}=
\frac{1}{x(c,c+3-r)}
\frac{T_{c+3-r,1,L}}{T_{c+3-r,2,L}}
\frac{T_{c+2-r,2,L}}{T_{c+2-r, 1,L}}.
\end{equation}.

 Combining \eqref{104B}-\eqref{104C} with \eqref{104A} yields
 \begin{equation}\label{104F}
x(c,r) \cdot z(r-1,2) = \frac{1}{x(c,c+3-r)} \frac{T_{c+3-r, 1,L}}{T_{c+3-r,2,L}}
\frac{T_{c+2-r,2,L}}{T_{c+2-r, 1,L}}.
 \end{equation}

We now proceed inductively. For $r=3,$ by  Proposition \ref{lem:horizontaledgefactor}(a),
$\frac{T_{c+3-r, 1,L}}{T_{c+3-r,2,L}} = y(c,2).$ Combining this equation for $y(c,2)$  with \eqref{104F} and \eqref{104E}, we conclude that  $\frac{T_{c-1, 2,L}}{T_{c-1,1,L}}=
y(c-1,2).$ Letting $r=4,$ we combine this equation for $y(c-1,2)$ with \eqref{104F}, and \eqref{104E} and conclude 
$\frac{T_{c-2, 2,L}}{T_{c-2,1,L}}=
y(c-2,2).$ Letting $r=5$ we combine this equation for $y(c-2,2)$ with  \eqref{104F}, and \eqref{104E} and conclude 
$\frac{T_{c-2, 2,L}}{T_{c-2,1,L}}=
y(c-3,2).$ Thus, proceeding inductively, we prove  Proposition \ref{lem:horizontaledgefactor}(b). (Note: For every $r$ in the upper half, $c-r$ is in the lower half. So we only need to use this inductive argument while $r$ is in the upper half.) \\

\textbf{Proof of Part (c).} Let $d > 2.$ We may inductively assume that i) the horizontal edge ratio relationship holds for diagonal $d-1,$ and ii) 
$x(c,r) \cdot z(r-1,d-1) = \frac{T_{c,r,d-1}}{T_{c,r-1,d-1}},$  
the base case of this induction assumption proven in Part b) for $d=2.$ 

For the proof we must first verify
 
\begin{equation}\label{504A}  
x(c, r) z(r - 1, d - 1)
\frac{y(r, d)}{y(r - 1, d)} = 
 1/(x(c, c + d + 1 - r) z(c + d - r, d - 1)) 
 \frac{y(c + d - r, d)}
   {y(c + d + 1 - r, d)}. 
\end{equation}
\normalsize
By slide symmetry,
\begin{equation}\label{504B}
\frac{T_{r,d,L}}{T_{r-1,d,L}} = \frac{T_{c+d-r,d,L}}{T_{c+d+1-r,d,L}}.
\end{equation} 
Using results previously proven in this section, the left-hand side of \eqref{504B} satisfies
\begin{equation}\label{504C}
\frac{T_{r,d,L}}{T_{r-1,d,L}} = \frac{T_{r,d-1,L}}{T_{r-1,d-1,L}} 
\frac{T_{r,d,L}}{T_{r,d-1,L}} \frac{T_{r-1,d-1,L}}{T_{r-1,d,L}} =x(c,r) z(r-1,d-1) y(r,d) \frac{1}{y(r-1,d)}.
\end{equation}
 
Using the induction assumption, the right-hand side of \eqref{504B} satisfies
\begin{multline}\label{504D}
\frac{T_{c+d-r,d,L}}{T_{c+d+1-r,d,L}} =
\frac{T_{c+d-r,d-1,L}}{T_{c+d+1-r,d-1,L}} 
\frac{T_{c+d-r,d,L}}{T_{c+d-r,d-1,L}} 
\frac{T_{c+d+1-r,d-1,L}}{T_{c+d+1-r,d,L}}
= \\ \frac{1}{x(c,c+d-r) z(c+d-1-r,d-1))} \frac{T_{c+d-r,d,L}}{T_{c+d-r,d-1,L}} 
\frac{T_{c+d+1-r,d-1,L}}{T_{c+d+1-r,d,L}}
\end{multline}
Combining \eqref{504C}-\eqref{504D} with  \eqref{504B} we obtain,
\begin{equation}\label{504E}
x(c,r) z(r-1,d-1) y(r,d) \frac{1}{y(r-1,d)} =\frac{1}{x(c,c+d-r) z(c+d-1-r,d-1))} \frac{T_{c+d-r,d,L}}{T_{c+d-r,d-1,L}} 
\frac{T_{c+d+1-r,d-1,L}}{T_{c+d+1-r,d,L}}.
\end{equation}
We now proceed inductively. If $r=d+1,$ by  Proposition  \ref{lem:horizontaledgefactor}(b),
$\frac{T_{c+d+1-r,d-1,L}}{T_{c+d+1-r,d,L}} = \frac{1}{y(c,d)}$. Combining this equation with 
\eqref{504E} and \eqref{504A} we infer
$\frac{T_{c+d-r,d,L}}{T_{c+d-r,d-1,L}} =y(c-1,d).$ If $r=d+2,$ then combining this equation for $y(c-1,d)$ with \eqref{504E} and \eqref{504A} implies $\frac{T_{c+d-r,d,L}}{T_{c+d-r,d-1,L}} =y(c-2,d).$ If $r=d+3,$ then combining this equation for $y(c-2,d)$ with \eqref{504E} and \eqref{504A} implies $\frac{T_{c+d-r,d,L}}{T_{c+d-r,d-1,L}} =y(c-3,d).$ Proceeding inductively, we prove that the horizontal edge ratio holds on diagonal $d,$ completing the inductive proof of   Proposition  \ref{lem:horizontaledgefactor}(c), implying completion of  the proof
of  Proposition  \ref{lem:horizontaledgefactor}, so that 
the proof of Theorem \ref{the:symmetryedgefactors} is complete.
 \end{proof}

\section{Proof of Theorem \ref{the:edgefactors}(ii), Case $i=c-1.$}\label{sec:sampleproof}

 This section proves 
Theorem \ref{the:edgefactors}(ii), for the case $i=c-1,$ the proof for $1 \le i \le c-2,$ as well as proof of the other assertions of Theorem \ref{the:edgefactors} being completed in the next section. Throughout this section,  we assume that $c \ge 5.$ and that  $T^{c,c-1}$ possesses  symmetry, justified by  Proposition \ref{lem:cgreaterthan4},  and Theorem \ref{the:isotropy} respectively.

Theorem \ref{the:edgefactors}(ii) requires proving  that all four edge ratio relationships are satisfied in $T^{c,c-1}.$   We however suffice with the proof for the base-edge ratio relationship, the proof of the others being similar and hence omitted. By  Proposition \ref{lem:upperlefthalf2}, to prove that $T^{c,c-1}$ satisfies the base-edge ratio relationship, it suffices to prove that the base edge ratio relationship holds in its upper left half. To prove satisfaction in the upper left half,  we prove that the base edge ratio relationship holds in all rows $r, 1 \le r \le c-1$ but only
in diagonals $d, 2 \le d \le  \lfloor \frac{c+2}{3} \rfloor,$ the proof for the $d=1$ case being highly similar and hence omitted.  In summary, in this section, we prove
\begin{equation}\label{equ:sampleproof}
		\frac{    T_{B,r,d}^{(c,c-1)}     } 
        {   T_{L,r,d}^{(c,c-1)}      }= 
        r_{B,L}(c-1,r,d), 
        \qquad 3 \le r \le c-2, 2 \le d \le r.
\end{equation}

Figure \ref{fig:sampleproofbase}  based on   \eqref{equ:baseedge} and the four edge ratio relationships, 
present the 9 edge-value arguments needed to compute 
$T^{c,c-1}_{B,r,d}.$ Similarly,
Figure  \ref{fig:sampleproofleft} based on 
\eqref{equ:leftedge} and 
the four edge ratio relationships, 
present the 9 edge-value arguments needed to compute
$T^{c,c-1}_{L,r,d}.$
 
The ``work" is done in the Figures. The proof of \eqref{equ:sampleproof}  is then simply completed by  plugging these arguments into the base and edge local functions and verifying \eqref{equ:sampleproof} is true.

  
\begin{figure}[!ht]
\begin{center}
\caption{The left-hand panel presents the edges of $T^{c}$ that are arguments of \eqref{equ:baseedge} used to compute   $T_{B,r,d}^{(c,c-1)}.$
The right-hand panel contains the values of these arguments.  
}\label{fig:sampleproofbase}
\begin{tabular} {|c|c|}
\hline
\begin{tikzpicture}[xscale=1,yscale=1]



\draw    (1,0)--(2,2)--(3,0)--(1,0);
\draw    (0,2)--(1,4)--(2,2)--(0,2);
\draw     (2,2)--(3,4)--(4,2)--(2,2);
\draw    (1,4)--(2,6)--(3,4)--(1,4);

\node [above] at (2.1,.1) {$T_{r+2,d+1}$};
\node [above] at (1.1,2.1) {$T_{r+1,d}$};
\node [above] at (3.1,2.1) {$T_{r+1,d+1}$};
\node [above] at (2.1,4.1) {$T_{r,d}$};

\node [left] at (1.5,5) {A};
\node [right] at (2.5,5) {B};
\node [below] at (2,4) {$C$};

\node [left] at (.5,3) {$D$};
\node [right] at (1.5,3) {$E$};
\node [left] at (2.5,3) {$F$};
\node [right] at (3.5 ,3) {$G$};
\node [below] at (1,2) {$H$};
\node [below] at (3,2) {$I$};

\node [left] at (1.5,1)  {$J$};
\node [right] at (2.5,1)  {$K$};
\node [below] at (2,0) {$L$};

\end{tikzpicture}
&

\begin{tikzpicture} [xscale=1,yscale=1]

 \node [right] at (1,7) {$A = \prod_{i=2}^r x(c,i) \cdot \prod_{i=2}^d y(r,i)$};
\node [right] at (1,6.4) {$B= A \cdot r_{R,L}(r,d)$};
\node [right] at (1,5.8) {$C=A \cdot r_{B,L}(c,r,d)$};
\node [right] at (1,5.2) {$D=A \cdot x_{c,r+1} \cdot z_{r,d}$};
\node [right] at (1,4.6) {$E=D \cdot  r_{R,L}(r+1,d)$};
\node [right] at (1,4.0) {$F= D \cdot y_{r+1,d+1} $};
\node [right] at (1,3.4) {$G = F \cdot r_{R,L}(r+1,d+1)$};
\node [right] at (1,2.8) {$H = D \cdot   r_{B,L}(c,r+1,d) $};
\node [right] at (1,2.2) {$I = F \cdot r_{B,L}(c,r+1,d+1)$};
\node [right] at (1,1.6) {$J= F \cdot x_{c,r+2} \cdot z_{r+1,d+1} $};
\node [right] at (1,1.0) {$K=J \cdot r_{R,L}(r+2,d+1)$};
\node [right] at (1,0.4) {$L=J \cdot r_{B,L}(c,r+2,d+1) $};

\end{tikzpicture} \\
\hline
\end{tabular}  
\end{center}

\end{figure}

  
\begin{figure}[!ht]
\begin{center}
\caption{The left-hand panel presents the edges of $T^{c}$ that are arguments of \eqref{equ:leftedge} used to compute   $T_{B,r,d}^{(c,c-1)}.$
The right-hand panel contains the values of these arguments. 
}\label{fig:sampleproofleft}
\begin{tabular} {|c|c|}
\hline
\begin{tikzpicture}[xscale=1,yscale=1]



\draw    (1,0)--(2,2)--(3,0)--(1,0);
\draw    (0,2)--(1,4)--(2,2)--(0,2);
\draw     (2,2)--(3,4)--(4,2)--(2,2);

\node [above] at (2.1,.1) {$T_{r+1,d}$};
\node [above] at (1.1,2.1) {$T_{r,d-1}$};
\node [above] at (3.1,2.1) {$T_{r,d}$};

\node [left] at (.5,3) {$M$};
\node [right] at (1.5,3) {$N$};
\node [left] at (2.5,3) {$O$};
\node [right] at (3.5 ,3) {$P$};
\node [below] at (1,2) {$Q$};
\node [below] at (3,2) {$R$};

 \node [left] at (1.5,1) {S};
\node [right] at (2.5,1) {T};
\node [below] at (2,0) {$U$};

\end{tikzpicture}
&

\begin{tikzpicture} [xscale=1,yscale=1]

 \node [right] at (1,7) {$M = \frac{A}{y(r,d)}$};
\node [right] at (1,6.4) {$N= M \cdot r_{R,L}(r,d-1)$};
\node [right] at (1,5.8) {$O=M \cdot y(r,d)$};
\node [right] at (1,5.2) {$P=O \cdot r_{R,L}(r,d)$};
\node [right] at (1,4.6) {$Q=M \cdot  r_{B,L}(c, r,d-1)$};
\node [right] at (1,4.0) {$R= O \cdot r_{B,L}(c, r,d) $};
\node [right] at (1,3.4) {$S = O \cdot x(c,r+1) \cdot z(r,d)$};
\node [right] at (1,2.8) {$T = S \cdot   r_{R,L}(r+1,d) $};
\node [right] at (1,2.2) {$U = S \cdot r_{B,L}(c,r+1,d)$};

\end{tikzpicture} \\
\hline
\end{tabular}  
\end{center}

\end{figure}


\section{Completion of proof of Theorem \ref{the:edgefactors}}\label{sec:completion}

Section \ref{sec:sampleproof} proves Theorem \ref{the:edgefactors}(ii) for the special case,  $i=c-1.$ This section provides proofs for the cases $1 \le i \le c-2.$ This section also proves Theorem \ref{the:edgefactors}(i). The proofs will be accomplished by a sequence of  Propositions.  Proposition \ref{lem:CTemp1} was illustrated and motivated by   \eqref{equ:tailratios}. 

\begin{proposition}\label{lem:CTemp1}
\begin{equation}\label{CTemp1}
T^{c,c-1}_{X,r,d}= g(c) T^{c-1}_{L,r,d}, \qquad X \in \{L,R,B\}, 1 \le r \le c-1, 1 \le d \le r.\end{equation}
\end{proposition}
\begin{proof}
For given $c$ we can compute $T^{c,c-1}_{L,1,1}$ by \eqref{equ:perimeteredge} and then  verify 
 that $g(c)=T^{c,c-1}_{L,1,1}$.
By Definition \ref{def:initialcgrid}, $T^{c-1}_{L,1,1} = 1,$ and therefore using Definition \ref{def:rr} or the four local functions
\begin{equation}\label{CTemp2}
T^{c,c-1}_{L,1,1}= g(c) T^{c-1}_{L,1,1}.\end{equation}

Both $T^{c-1}$ and $T^{c,c-1}$ satisfy the four edge-ratio relationships, $T^{c-1}$ by definition, and $T^{c,c-1}$ by the result of Section \ref{sec:sampleproof}. Therefore,  \eqref{CTemp1} follows from \eqref{CTemp2} because all edges are identical multiples (using the four edge-ratio relationships) of the left edge of the top corner triangle. 
\end{proof}

\begin{proposition}\label{lem:CTemp2}
\begin{equation}\label{CTemp3}
T^{c,i}_{X,r,d}=  f(c,i) \times T^{i}_{X,r,d}, \qquad
1 \le r \le c-i, 1 \le d \le r, X \in \{L,R,B\}.\end{equation}
\end{proposition}
\begin{proof}
We  iterate \eqref{CTemp1}.
 
If we start with an initial $c-1$ grid and row reduce it once, then by  Proposition \ref{lem:CTemp1} we have
$T^{c-1,c-2}_{L,1,1}= g(c-2) T^{c-2}_{L,1,1}.$
But row reducing $T^{c}$ once and then row reducing $T^{c,c-1}$ once is the same
as row reducing $T^{c}$ twice. In other words, 
$T^{c,c-2}_{L,1,1}= g(c) \times g(c-1) \times T^{c-2}_{L,1,1}$ and hence, since
$T^{c,c-2}$ and $T^{c-2}$ both satisfy the edge-ratio relationships, 
$T^{c,c-2}_{X,r,d}= g(c) \times g(c-1) \times T^{c-2}_{X,r,d}$ for $1 \le r \le c-2, 1 \le d \le r, X \in \{L,R,B\}.$ 

Inductively iterating the above argument we have, for $1 \le i \le c-1,$
that  $T^{c,c-i}_{L,1,1}= \prod_{j=0}^{j=i-1}g(c-j) \times T^{c-i}_{L,1,1}=   f(c,c-i) \times T^{c-i}_{1,1,L}$ the last equality following  from  Proposition \ref{lem:gf}.
Since, $T^{c,c-i}$ and $T^{c-i}$ both satisfy the edge-ratio relationships,
we then also have $T^{c,c-i}_{X,r,d}=  f(c,i) \times T^{c-i}_{X,r,d}$
for $1 \le r \le c-i, 1 \le d \le r, X \in \{L,R,B\}.$
\end{proof} 

 Propositions \ref{lem:CTemp1} and \ref{lem:CTemp2}, complete the proof of Theorem \ref{the:edgefactors}(ii). We proceed to prove Theorem \ref{the:edgefactors}(i).

\begin{proposition}\label{lem:CTemp3}
$$t(c,i) = f(c,i) \Delta(1,1,r_{B,L}(i,1,1))$$
\end{proposition}
\begin{proof} 
By Definition \ref{def:tail},  
the tail resistance distance $t(c,i)$  is simply the 12 o'clock leg of the  $Y$ arising from applying the $\Delta-Y$ transformation to   $T^{c,i}_{1,1}.$   By  Proposition \ref{lem:CTemp2}, $T^{c,i}_{1,1,X}= f(c,i) \cdot T^{i}_{1,1,X}, \text{ for } X \in \{L,R,B\}$. By Definition \ref{def:initialcgrid},
$\{
T^{i}_{1,1,L}, 
T^{i}_{1,1,R},
T^{i}_{1,1,B}\}=
\{1,1,r_{B,L}(i,1,1)\}.$
The proof now follows by substitution and verification.
\end{proof}

\begin{proposition}\label{lem:CTemp4}
i) $t(c,1) = \frac{c}{2c+1}$\\
ii) $t(c,i) = \frac{1}{i} \frac{c}{2c+1}.$\\
iii) $\frac{t(c,i)}{t(c,1)} = \frac{1}{i}.$
\end{proposition}
\begin{proof} By  Proposition \ref{lem:CTemp3} and a straightforwad verification. \end{proof}

But  Proposition \ref{lem:CTemp4} is Theorem \ref{the:edgefactors}(i).  Therefore, the proof of Theorem \ref{the:edgefactors} is complete.

 
 

 \bmhead{Acknowledgments}{The author acknowledges A. Francis for suggesting exploring the triangular grids; A. Francis and E. Evans, for hosting a webinar introducing me to their paper and its results; E. Evans, for several follow-up papers wherein definitions and notations  have been considerably improved   resulting with greater clarity; and several reviewers  (anonymous and otherwise) who suggested clarifying improvements to definitions and notations.}

\section*{Declarations}
\begin{itemize}
\item Funding: No funding has been received for this research.
\item Conflict of interest/Competing interests: There are no conflicts of interest or competing interests for this article. 
\item Ethics approval: Not applicable
\item Consent to participate: Not applicable:
\item Consent for publication: Not applicable (There is only one author)
\item Availability of data and materials: All data is included in the paper.
\item Code availability: No special code is used. 
\item Authors' contributions: There is only one author who wrote the work. Various improvements to notations and defintions have been attributed to published works as applicable.
\end{itemize}

\end{document}